\documentclass[10pt]{articleHJ}
\usepackage[centertags]{amsmath}
\usepackage{caption,graphicx}
\usepackage{subcaption}
\usepackage[text={6.0in,8.6in},centering,letterpaper]{geometry}
\usepackage[numbers,comma,square,sort&compress]{natbib}
\usepackage[utf8]{inputenc}
\usepackage{mathrsfs}
\usepackage{amsthm, amsfonts, mathrsfs, amsfonts, amssymb}
\usepackage[british]{babel}
\usepackage[export]{adjustbox}
\usepackage{mathtools}
\usepackage{comment}
\usepackage{url} 
\usepackage{xcolor}
\everymath={\displaystyle}
\numberwithin{equation}{section}

\newcommand{\R}{\ensuremath{\mathbb{R}}}

\renewcommand{\b}{\ensuremath{\beta}}
\newcommand{\e}{\ensuremath{\varepsilon}}
\newcommand{\p}{\ensuremath{\partial}}
\renewcommand{\d}{\ensuremath{\delta}}

\newcommand{\G}{\ensuremath{\Gamma}}
\newcommand{\s}{\ensuremath{\sigma}}

\renewcommand{\a}{\ensuremath{\alpha}}

\renewcommand{\L}{\ensuremath{\mathcal{L}}}
\renewcommand{\O}{\ensuremath{\mathcal{O}}}
\DeclarePairedDelimiter\abs{\lvert}{\rvert}

\DeclarePairedDelimiter{\ip}\langle\rangle
\DeclarePairedDelimiter{\bip}{\big\langle}{\big\rangle}
\DeclarePairedDelimiter{\nrm}\lVert\rVert
\DeclarePairedDelimiter{\norm}\lVert\rVert

\DeclareMathOperator{\tr}{tr}
\newcommand{\eb}{\ensuremath{\mathcal{E}_B}}

\def\Re{\mathop\mathrm{Re}\nolimits}			
\newcommand{\sref}[1]{(\ref{#1})}                       

\newtheorem{thm}{Theorem}[section]
\newtheorem{cor}[thm]{Corollary}
\newtheorem{lem}[thm]{Lemma}
\newtheorem{prop}[thm]{Proposition}

\newtheorem*{defn*}{Definition}

\newlength{\defbaselineskip}
\newcommand{\setlinespacing}[1]%
           {\setlength{\baselineskip}{#1 \defbaselineskip}}

\newcommand{\singlespacing}{\setlength{\baselineskip}{\defbaselineskip}}

\begin{document}

\begin{frontmatter}
\title{Stability of Traveling Waves on Exponentially Long Timescales in Stochastic Reaction-Diffusion Equations}
\journal{SIADS}

\author[LD1]{C. H. S. Hamster\corauthref{coraut}},
\corauth[coraut]{Corresponding author. }
\author[LD2]{H. J. Hupkes},
\address[LD1]{
  Mathematisch Instituut - Universiteit Leiden \\
  P.O. Box 9512; 2300 RA Leiden; The Netherlands \\
  Email:  {\normalfont{\texttt{c.h.s.hamster@math.leidenuniv.nl}}}
}
\address[LD2]{
  Mathematisch Instituut - Universiteit Leiden \\
  P.O. Box 9512; 2300 RA Leiden; The Netherlands \\ Email:  {\normalfont{\texttt{hhupkes@math.leidenuniv.nl}}}
}

\date{\today}

\begin{abstract}
\singlespacing
In this paper we establish the meta-stability of traveling waves for a class
of reaction-diffusion equations forced by a multiplicative noise term. In
particular, we show that the phase-tracking technique developed
in \cite{hamster2017,hamster2020} can be maintained over timescales that are exponentially long with respect to the noise intensity. This is achieved
by combining the generic chaining principle with a mild version
of the Burkholder-Davis-Gundy inequality to establish logarithmic 
supremum bounds for  stochastic convolutions in the critical regularity regime.
\end{abstract}



\end{frontmatter}

\section{Introduction}
\label{p4:sec:int}

In this paper we focus on the stochastic Nagumo equation
\begin{align}
\begin{split}
\label{p4:eq:mr:main:spde}
dU &=
  \left[ \rho\mspace{1mu} \partial_{xx} U + f(U)\right] dt
  + \sigma g(U) d W^Q_t,\\
\end{split}
\end{align}
in which we take $U=U(x,t)$ with $x \in \R$ and $t \geq 0$. The nonlinearities
are given by
\begin{equation}
    f(u)=u(1-u)(u-a),
    \qquad \qquad
    g(u)=u(1-u)\chi(u)
\end{equation}
for a parameter $a \in (0,1)$ and a smooth cut-off function $\chi(u)$ that forces $g$ to be bounded and globally Lipschitz continuous on $\R$.
The stochastic forcing is generated by the cylindrical $Q$-Wiener process $W^Q_t$ characterized by the convolution operator
\begin{equation}
Q: L^2(\R) \to L^2(\R), \qquad \qquad
[Qv](x) =  \int_{-\infty}^\infty e^{-(x-y)^2} v(y) \, d y.
\end{equation}
In particular, our noise satisfies
the formal relation
\begin{equation}
    E \Big[  d W^Q_s(x_0)d W^Q_t(x_1) \Big] = \delta(t - s) e^{-(x_0 - x_1)^2}
\end{equation}
and hence is white in time but colored and translationally invariant in space.
The well-posedness of such equations has been studied extensively
\cite{liurockner,concise} and one can construct globally defined solutions
in (for example) the affine space \cite[Prop. 5.2]{hamster2020}
\begin{equation}
\mathcal{U}_{H^1} = H^1(\R) + \frac{1}{2}\big( 1 - \tanh(\cdot) \big).
\end{equation}

The choice for this space is motivated by the fact that it contains
the well-known deterministic traveling wave solution
\begin{equation}
\label{p4:eq:int:def:det:trv:wave}
    U(x,t) = \Phi_0( x - c_0 t), \qquad \qquad \Phi_0(-\infty) = 1, \qquad \Phi_0(+\infty) = 0
\end{equation}
for \sref{p4:eq:mr:main:spde} with $\sigma = 0$. In \cite{hamster2017,hamster2020} we showed that this pair $(\Phi_0,c_0)$ can be generalized to a branch of so-called \textit{instantaneous stochastic waves} 
$(\Phi_{\sigma},c_{\sigma})$ for \sref{p4:eq:mr:main:spde} 
that - at onset - travel with velocity velocity $c_{\sigma}$ and 
feel only stochastic forcing. These waves can be shown to satisfy
\begin{equation}
     \norm{ \Phi_{\sigma} -\Phi_0}_{H^2}+\abs{c_{\sigma} - c_0} = \O (\sigma^2).
\end{equation}

 The key question is if one can understand the perturbations
\begin{align}
\label{p4:eq:int:perturbations}
 V(t) =  U(\cdot +\G(t),t) - \Phi_\s
\end{align}
from these profiles, using an appropriate phase shift $\Gamma$ (which we will define in \S5) to stochastically `freeze' the solution $U$. 
In particular, we are interested in the behavior of the stopping time
\begin{equation}
\label{p4:eq:mr:defStopTime}
    t_{\mathrm{st}}(\eta)
     = \inf \{
       t \geq 0:
         \norm{V(t)}_{L^2}^2 + \int_0^t e^{-\e(t-s)} \norm{V(s)}_{H^1}^2 \, ds > \eta
        \},
\end{equation}
which measures when $U$ exits an appropriate orbital $\eta$-neighborhood of the profile $\Phi_{\sigma}$. We invite the reader to think of $\varepsilon > 0$ as a small external parameter that is required for regularity purposes 
but kept fixed throughout the entire paper. 

Our main result states that this exit-time is (with probability exponentially close to one) exponentially long with respect to the parameter $1/\sigma$.
As such, it establishes the meta-stability of
the deterministic
traveling wave \sref{p4:eq:int:def:det:trv:wave} under
small stochastic forcing, significantly extending
our earlier results in \cite{hamster2017,hamster2020}.
We refer to the end of {\S}\ref{p4:sec:int} for 
a characterization of the class of systems for which our bound
is valid, together with comments concerning improvements and generalizations.

\begin{thm}
\label{p4:thm:mr}
Fix\footnote{We emphasize that the subsequent constants $\delta_{\eta}$, $\delta_{\sigma}$ and $\kappa$ all tend to zero as $\e \downarrow 0$. This makes sense, since the expression in \sref{p4:eq:mr:defStopTime} will naturally grow with $t$ in this limit, making it useless as a measure for stability.}
a sufficiently small constant
$\e > 0$. Then there exist
constants $\delta_{\eta} > 0$, 
$\delta_{\sigma}> 0$ and $0< \kappa <1$ so that the following holds true.
For  any 
$0< \eta < \delta_{\eta}$,
any $0 < \sigma < \delta_{\sigma}$
and any $U(0) \in \mathcal{U}_{H^1}$
that satisfies $\norm{U(0) - \Phi_{\sigma}}_{H^1}^2 < \kappa \eta$,
there exists a scalar stochastic process $\Gamma$ so that 
 \begin{equation}
   \label{p4:eq:mr:estimate:on:prob}
     P( t_{\mathrm{st}}(\eta) < T ) \leq 
     2 T \mathrm{exp}\big[- \frac{\kappa \eta}{\sigma(\sigma + \sqrt{\eta})}\big]
\end{equation}
holds for all $T \ge 2$.
\end{thm}

We remark here that general `exit-problems' 
have been well-studied in finite-dimensional contexts \cite{day1990large,freidlin1998random},
chiefly based on Freidlin-Wentzell theory (see e.g. 
\cite[Ch. 5]{dembo2011large}). However,
due to technical challenges much less is known in infinite dimensions \cite{gautier2005uniform,berglund2013}. The recent paper by Salins and Spiliopoulous \cite{salins2019} discusses some of the main developments
in this area, which chiefly focus on SPDEs with gradient-independent noise posed on finite domains. 
The authors extensively discuss several key complications
that arise in infinite-dimensional settings,
such as the construction of controlled paths,
the regularity of the quasipotential and
the non-compactness of the domain of attraction.
Nevertheless,  they 
were able to establish tightness results  
that lead naturally to large deviation principles \cite{cerrai2004large}
by exploiting the fact that 
the associated semigroups are compact.
In this fashion, several of the key Freidlin-Wentzell results concerning the exit-time and exit-shape could be recovered.

In the context of traveling waves,
Bouard and Gautier \cite{de2008exit} showed that 
the classic soliton solutions to the KdV equation survive on timescales that are algebraically long with respect to the noise strength. In this analysis they are aided by the dispersive character of the system and explicit Lyapunov functionals and conserved quantities. In addition, MacLaurin and Bressloff
established bounds similar to
\sref{p4:eq:mr:estimate:on:prob}
for waves in a  neural field model that is posed on a ring and that does not feature spatial derivatives \cite{maclaurin2020}. A key feature in their approach is an immediate contractivity condition that is typically hard to check (or enforce) in general models; see \cite[\S1]{hamster2020} for a detailed discussion.

The special features described above do not apply in the current setting and we take a completely different approach. In particular, we focus directly
on the underlying stochastic convolution integrals and obtain
logarithmic growth bounds, which translate readily into exponentially long exit-times via an exponential Markov inequality. The main issues that we have to face
concern the regularity of these stochastic convolutions, which require
specialized tools and a careful choice of norms. For example, we need
to exploit the fact that our semigroup admits an $H^\infty$-calculus to keep the $L^2$-bound of $V(t)$ under control. In addition, our problem 
only provides integrated rather than pointwise control on the $H^1$-norm of $V(t)$, which forces us to include the non-standard integral in \sref{p4:eq:mr:defStopTime}. Our ability to control the $H^1$-norm in this fashion is a clear distinguishing feature of our approach and offers several important benefits that we discuss in the sequel.

Naturally, we intend to continue work to streamline our techniques with 
the more classic Freidlin-Wentzell approach. Indeed, we view this paper merely as a proof-of-concept to show that wave-tracking over exponentially long time-scales is possible for a broad class of stochastic reaction-diffusion PDEs. To enhance the readability, we have stated our results for the simple scalar system \sref{p4:eq:mr:main:spde} and refrained from 
sharpening 
the bound \sref{p4:eq:mr:estimate:on:prob}
to the fullest extent attainable by our approach.

We do wish to point out that 
at some point in any typical stability analysis, the stochastic integrals need to be addressed. For example, in \cite{salins2019} factorization is used
to obtain preliminary bounds on 
the pertinent convolution integrals, but
this is not possible in our case as we explain below. Since sharp bounds on integrals play such an important role in deterministic stability theory, we feel that estimation techniques similar to those that we develop here could play a strategic role when combining the fields of probability and pattern formation. Even outside of this specific scope, the ideas 
and references in {\S}\ref{p4:sec:prlm}-\ref{p4:sec:supb} could be useful for researchers
interested in quantitative procedures to characterize growth rates of stochastic processes.

\paragraph{Stochastic waves}
The impact of noise on pattern formation is an important topic that has attracted
significant interest from the applied community
\cite{Armero1996,garciaspatiallyextended,schimansky1991,vinals1991numerical,shardlow,hallatschek},
but for which little rigorous mathematical theory is available \cite{Bloemker,kuske2017,mueller1995,hausenblas2020theoretical}. 
The Nagumo equation is a natural starting point for such investigations since it has served in the past as a prototypical system to analyze the interaction between two competing stable states in spatially extended domains \cite{Aronson1975nonlinear,AW78}. The deterministic traveling waves 
\sref{p4:eq:int:def:det:trv:wave} represent a primary invasion mechanism by which the favorable state can spread throughout the entire domain. They are robust under perturbations, which allows them to be used as building blocks to understand the global behavior of \sref{p4:eq:mr:main:spde} in one \cite{fife1977,volpert1994traveling,kapitula}
but also higher spatial dimensions \cite{bhm,kap1997,matano2019asymptotic}.

The behavior of these invasion waves under several types of stochastic forcing has been studied by various authors using a range of different techniques.
The consensus emerging from a number of formal computations
for \sref{p4:eq:mr:main:spde}
is that
- to leading order in $\sigma$ - the phase-shift of the wave follows a Brownian motion with a variance
that can be expressed in closed form \cite{Bressloff,garciaspatiallyextended,cartwright2019}.
Various rigorous approaches have been pursued
over the past five years that can
successfully explain this diffusive behavior on short time scales
\cite{stannatnag,stannatbistable,stannatkruger,Inglis};
see e.g \cite{kuehnreview} for a very recent overview
and the introductions of \cite{hamster2017,hamster2020} for a
detailed technical discussion.  Several of these techniques
have been extended to stochastic neural field equations
\cite{langstannat2016l2,bressloff2015nonlinear,maclaurin2020}
and (very recently) to the FitzHugh-Nagumo system \cite{gnann}.

In a recent series of papers \cite{hamster2017,hamster2018uneq,hamster2020},
we pioneered a novel `stochastic freezing' approach to rigorously analyze the behavior of traveling fronts and pulses to a large class of reaction-diffusion equations (RDEs) - which includes \sref{p4:eq:mr:main:spde} and the (fully diffusive) FitzHugh-Nagumo system.  In essence, we developed a stochastic version of the freezing approach introduced by
Beyn \cite{beyn2004freezing}, which allows us to adopt the spirit behind the modern machinery for deterministic stability
issues initiated by Howard and Zumbrun
\cite{zumbrunhoward}. The power of this approach is that it leads naturally
to \textit{long-term} predictions concerning both the speed and the shape of the 
stochastic wave that can be computed to arbitrary order in $\sigma$. We demonstrated the accuracy of these novel predictions in \cite{hamster2020} by performing a series of numerical experiments. As a consequence,
we now have a quantitative explanation
for the wave-steepening and speed-reduction phenomena that
were illustrated numerically in \cite{lord2012}
and - in a special case - derived formally in \cite{cartwright2019}
using a collective coordinate approach. 

\paragraph{Regularity issues}
We do point out that the long-term stability results in \cite{maclaurin2020} discussed above
are based on an alternative phase-tracking mechanism
proposed in \cite{Inglis}, which predates our work. However,
the key novel feature of the approach in \cite{hamster2017,hamster2018uneq,hamster2020} is that the perturbation
$V$ in the decomposition \sref{p4:eq:int:perturbations} is measured
in the same reference frame as the frozen profile $\Phi_{\sigma}$. This allows
the delicate interaction between the speed and shape of the wave to be untangled,
but also presents several fundamental complications that need to be carefully addressed. The most important of these is that the stochastic phase shift
causes extra diffusive correction terms for $V$ that are not seen in the deterministic context, together with a multiplicative noise term 
that involves the derivative of $V$. Unlike any of the previous approaches in this area, we hence need to keep the $H^1$-norm of $V(t)$ under control.

To be more specific, an essential step in our stability proofs is to obtain
bounds for the expression
\begin{align}
\label{p4:eq:int:exL2}
    E\sup_{0\leq t\leq T}\nrm{\int_0^tS(t-s)B\big(V(s), \p_xV(s)\big)dW^Q_s}^2_{L^2},
\end{align}
together with its integrated $H^1$-counterpart
\begin{align}
\label{p4:eq:int:exH1}
        E\sup_{0\leq t\leq T}\int_0^te^{-\e(t-s)}\nrm{\int_0^sS(s-s')
         B\big(V(s'), \p_xV(s')\big)dW^Q_{s'}}^2_{H^1} \,ds.
\end{align}
Here, $B\big(V(s), \p_xV(s)\big)$ represents a suitable Hilbert-Schmidt operator and $S$ denotes the semigroup associated to the linearization
of \sref{p4:eq:mr:main:spde} around the deterministic traveling wave
\sref{p4:eq:int:def:det:trv:wave}. The precise
expression for $B$ can be found in \S5,
where we recap the computations from
our previous papers
\cite{hamster2017,hamster2020}; see \cite[\S2.2]{hamster2020} for an extensive informal explanation. In  \cite{hamster2017,hamster2018uneq}, we used
the mild Burkholder-Davis-Gundy (BDG) inequality 
obtained by Veraar \cite{veraar2011note} to control \sref{p4:eq:int:exL2},
but the resulting bounds are unfortunately not optimal. As such,
they restricted the validity range of our rigorous results to timescales of order $T \sim \sigma^{-2}$.

This shortfall is repaired by the bound in Theorem \ref{p4:thm:mr}, which 
confirms that our phase-tracking can be maintained over the exponentially long timescales observed in the numerical results 
from \cite{hamster2020}.
We emphasize that our improved bound also covers regimes where the stochastic phase $\Gamma$ is expected to be very far away from its deterministic counterpart. This provides a solid theoretical underpinning to the formal
predictions that we made in \cite{hamster2020} concerning
the stochastic corrections to  \sref{p4:eq:int:def:det:trv:wave}.
In addition, it allows us to obtain stochastic meta-stability results under the same spectral conditions required for deterministic stability.

To understand the issues that are involved, it is highly instructive to consider the scalar Ornstein-Uhlenbeck process 
\begin{align}
\label{p4:eq:int:def:OU}
X(t)=\int_0^te^{-(t-s)}d\beta_s,
\end{align}
which here starts at $X(0) =0$ and 
is driven by a standard Brownian motion $\beta_t$. Since $B(0,0) \neq 0$,
the behavior of $X$ is highly comparable to that of $V$ at lowest order in $\s$. 
Indeed, the deterministic dynamics pulls $X$ towards zero at an exponential rate, but the stochastic forcing does not vanish there. Applying
the mild Burkholder-Davis-Gundy inequality to \sref{p4:eq:int:def:OU}
results in the bound
\begin{equation}
    E \sup_{0 \leq t \leq T} \abs{X(t)}^2 \leq K  \int_0^T 1 \, ds = KT.
\end{equation}
This hence fails to reproduce the well-known fact that
this expectation behaves as $\O\big( \ln(T)\big)$ for large $T$,
which was originally established by examining crossing numbers \cite{pickands1969}
or analyzing explicit probability distributions \cite{alili2005representations}.
Fortunately, a more general abstract approach has been developed in recent years.

\paragraph{Chaining}
A powerful modern tool to derive supremum bounds for stochastic processes is commonly referred
to as `generic chaining'; see \cite{talagrand} for an accessible introduction.\footnote{This unpublished chapter by Pollard could also be useful: \url{http://www.stat.yale.edu/~pollard/Books/Mini/Chaining.pdf}}
Based on contributions from a range of authors, including Kolmogorov, Dudley, Fernique and Talagrand, it uses information on the increments of a stochastic
process to establish long-term supremum bounds. For instance, exploiting the fact
that the Ornstein-Uhlenbeck process \sref{p4:eq:int:def:OU} is centered and Gaussian, one can obtain the tail bound
\begin{align}
\label{p4:eq:int:tailb}
    P(|X(t)-X(s)|> \vartheta)\leq 2 e^{-\frac{\vartheta^2}{2d(t,s)^2}}
\end{align}
characterized by the metric $d(t,s) = \sqrt{E \big(X(t)-X(s)\big)^2}$.
An explicit computation yields the bound
\begin{equation}
\label{p4:eq:int:comp:d:ou}
\begin{array}{lcl}
  d(t,s)^2 
    &=& 
\frac{1}{2}\big( 2-e^{-2t}-e^{-2s}-2(e^{-|t-s|}-e^{-(t+s)}) \big)
\\[0.2cm]
&\leq &1-e^{-|t-s|}
\\[0.2cm]
& \leq & \min\{ \abs{t - s} , 1 \}.
\end{array}    
\end{equation}
This shows that the covering number $N(T,d,\nu)$ - which measures the
minimum number of intervals of length $\nu$ or less in the metric $d$ required to cover $[0,T]$ - can be bounded by $T/\nu^2$ for $\nu \in (0,1]$ and by 1 for $\nu \geq 1$. The main result in \cite{talagrand} 
- see Theorem \ref{p4:thm:supb:talagrand} below - now provides the Dudley bound
\begin{equation}
    \sup_{t\in[0,T]}X_t
    \sim \int_0^\infty\sqrt{\ln(N(T,d,\nu))}d\nu
    \leq
   \int_0^1\sqrt{\ln\left(T/\nu^2\right)}d\nu
   \sim \sqrt{ \ln(T)},
\end{equation}
which captures the desired logarithmic behavior
in a relatively straightforward manner.

Our main contribution in this paper is that we extend this technique to provide
similar sharp bounds for the stochastic integrals
\sref{p4:eq:int:exL2} and \sref{p4:eq:int:exH1}. 
On account of the regularity
issues that are involved, this is a surprisingly delicate task.
In fact, we are not aware of any related results in this direction
besides the factorization method developed by Da Prato, Kwapie\'n and Zabczyk \cite{dapratoregularity}, which typically only provides polynomial 
bounds in $T$. Let us remark that it was not immediately clear
to us how this factorization technique should be applied in the present setting, 
because it introduces extra singularities into integrals that cannot 
be readily accommodated in our critical regularity regime.

\paragraph{Obstructions}
In order to illustrate the key complications, let us consider the
$L^2$-valued process
\begin{align}
\label{p4:eq:int:infdimou}
 Y(t)=\int_0^tS(t-s)BdW^Q_s ,
\end{align}
which can be seen as an infinite-dimensional version of 
\sref{p4:eq:int:def:OU}. Here $B$ is an appropriate constant Hilbert-Schmidt operator, which can be used to define the covariance operator
\begin{align}
    Q_\infty=\lim_{t\to\infty}\int_0^tS(t-s)BQB^*S^*(t-s)ds.
\end{align}
The analogue of the bound \sref{p4:eq:int:comp:d:ou}  
is now given by\footnote{This computation can be made rigorous using \cite[\S5]{hairernotes}.} 
\begin{align}
\label{p4:eq:int:def:d:sq:inf:ou}
    d(t,s)^2=E\nrm{Y(t)-Y(s)}_{L^2}^2
    \leq 2\tr\big((I-S(t-s))Q_\infty\big),
\end{align}
but this time there is no $\alpha > 0$ for which one can extract a term
of the form $|t-s|^\a$ from the difference $S(t-s) - I$. In principle this can be repaired by `borrowing' some smoothness from $B$, but in our case this would again lead to unintegrable singularities. 

In order to resolve this, it is crucial to combine the strong points of both the chaining technique and the mild Burkholder-Davis-Gundy inequality. Indeed, 
the former works well in the regime where $\abs{t - s} \geq 1$
in \sref{p4:eq:int:exL2}, since here the decay and smoothening properties of the semigroup can both be put to excellent use. On the other hand,
for $\abs{t-s} \leq 1$, the $H^\infty$-calculus underlying the mild Burkholder-Davis-Gundy inequality can resolve the critical regularity issues
associated to supremum bounds without causing too much growth.
The main issue is to set up an appropriate framework that allows this splitting
to be achieved.

The second fundamental complication is that the integrands in 
\sref{p4:eq:int:exL2} and \sref{p4:eq:int:exH1} are time-dependent, which means
that - in contrast to \sref{p4:eq:int:infdimou} - the stochastic integrals are not  Gaussian. In this case, one must construct a metric such that a corresponding tail bound like \sref{p4:eq:int:tailb} can be derived from scratch. Effectively, this requires us to control \textit{all} the moments of the increments of \sref{p4:eq:int:exL2}. This is made possible by an effective use of stopping times in combination with a mild It{\^o} formula.

\paragraph{Scope and outlook}

In order to make the arguments in this paper as clear and concise as possible, we chose to restrict our attention to the single specific problem \sref{p4:eq:mr:main:spde}. However, we emphasize that our arguments transfer
immediately - almost verbatim - to the general class of (multi-component) problems considered in \cite{hamster2017} and \cite{hamster2020}, with the single
restriction that all diffusion coefficients must be equal (condition (hA) in \cite{hamster2017}). This latter restriction can be removed
by applying the spirit of \cite{hamster2018uneq}, but this requires
more complicated machinery that we will describe in an extensive
forthcoming companion paper. There we will also address the long-term validity of 
the perturbation results from \cite{hamster2020}.
In addition, we show that the bound
\sref{p4:eq:mr:estimate:on:prob}
can be improved by eliminating the $\sqrt{\eta}$ term, which arises here as a consequence of a shortcut that we take to estimate \sref{p4:eq:int:exH1}.

\paragraph{Organization}
We start in \S\ref{p4:sec:prlm} by introducing some basic probabilistic and deterministic concepts. The heart of this paper is contained in \S\ref{p4:sec:supb}, where we provide logarithmic bounds for the stochastic integrals \sref{p4:eq:int:exL2}
and \sref{p4:eq:int:exH1}. Several supremum bounds for deterministic integrals
are provided in \S\ref{p4:sec:supbd}, which allow for a streamlined
proof of our main theorem in \S\ref{p4:sec:nls}.

\paragraph*{Acknowledgments}
HJH acknowledges support from the Netherlands Organization for Scientific Research (NWO) (grant 639.032.612). Both authors wish to thank two anonymous referees for helpful suggestions, which helped strengthen the bound
\sref{p4:eq:mr:estimate:on:prob} besides improving
the readability of the paper.

\section{Preliminaries}
\label{p4:sec:prlm}

In this section we collect several useful preliminary results that will streamline our arguments. We start in {\S}\ref{p4:sec:prlm:det}
by recalling well-known facts concerning the linearization of the Nagumo PDE around its traveling wave. We subsequently consider the relation between tail bounds and moment estimates for scalar stochastic processes
in {\S}\ref{p4:sec:prlm:mom:tail}. Finally, in {\S}\ref{p4:sec:prlm:sup} 
we formulate the key technical tools that will allow us to apply the chaining principle to stochastic convolutions
in the critical regularity regime.

\subsection{Semigroup bounds}
\label{p4:sec:prlm:det}
It is well-known that the Nagumo PDE \sref{p4:eq:mr:main:spde}
with $\sigma = 0$ admits a traveling front solution 
$U(x,t) = \Phi_0(x - c_0 t)$ that necessarily satisfies
the traveling wave ODE
\begin{equation}\label{p4:eq:MR:TWODE}
    \rho \Phi_0''+c_0 \Phi_0' + f(\Phi_0)=0,
    \qquad \qquad
    \Phi_0(-\infty) = 1, \qquad \Phi_0(+\infty) = 0.
  \end{equation}
  The associated linear operators
  \begin{equation}
   \label{p4:eq:mr:def:l:tw:and:adj}
   \mathcal{L}_{\mathrm{tw}} v =
      \rho v''+c_0 v'  + Df\left(\Phi_0\right) v,
      \qquad \qquad
      \mathcal{L}^*_{\mathrm{tw}} w 
      =
      \rho w''-c_0 w'  + Df\left(\Phi_0\right) w,
  \end{equation}
  which we view as maps from $H^2(\R)$ into $L^2(\R)$,
  both admit a simple eigenvalue at $\lambda  =0$ and have no other spectrum in the half-plane $\{\Re \lambda \geq -2\beta\} \subset \mathbb{C}$
  for some $ \beta > 0$. 
    Writing $P_{\mathrm{tw}}$ for the spectral projection onto this neutral eigenvalue for $\mathcal{L}_{\mathrm{tw}}$, we can obtain the identifications
    \begin{equation}
        \mathrm{Ker}(\mathcal{L}_{\mathrm{tw}} ) = \mathrm{span}\{ \Phi_0' \},
        \qquad \qquad
        \mathrm{Ker}(\mathcal{L}_{\mathrm{tw}}^* ) = \mathrm{span}\{ \psi_{\mathrm{tw}} \},
        \qquad \qquad
        P_{\mathrm{tw}} v = \langle v, \psi_{\mathrm{tw}} \rangle_{L^2} \Phi_0'
    \end{equation}
    by writing\footnote{Note that $\psi_\mathrm{tw}$ is in $L^2$, which can be shown by a direct computation or by using Sturm-Liouville theory \cite[Thm. 2.3.3]{kapitula}. If one wishes to enter the monostable regime by choosing $a<0$, then $\psi_\mathrm{tw}$ becomes unbounded, which 
    is typically accommodated by using weighted spaces. We stress that the explicit formula for $\psi_{\mathrm{tw}}$ is not used anywhere in this paper.}
    $\psi_{\mathrm{tw}}(\xi) = \kappa e^{-\frac{c_0 \xi}{\rho}}\Phi_0'(\xi)$ for some $\kappa$
    that we fix by the requirement $\langle \Phi_0', \psi_{\mathrm{tw}}   \rangle_{L^2} = 1$.
    
    In fact, the operator $\mathcal{L}_{\mathrm{tw}}$ is sectorial and hence generates an analytic semigroup $S(t) = e^{\mathcal{L}_{\mathrm{tw}} t}$;
    see \cite[Prop. 4.1.4]{lorenzi2004analytic} and \cite[Prop. 6.3.vi]{hamster2017}.
    Upon introducing the notation
    \begin{align}
\begin{split}
\label{p4:eq:prlm:def:j:tw}
\mathcal{J}_{\mathrm{tw}}(t)[v,w]=&\,
\ip{S(t) v, S(t)w}_{L^2}
+\frac{1}{2\rho}\ip{S(t) v,   (\mathcal{L}_{\mathrm{tw}}-\rho\p_{xx})S(t) w}_{L^2}
\\
    &
    +\frac{1}{2\rho}\ip{S(t) v,
    (\mathcal{L}^*_{\mathrm{tw}}-\rho\p_{xx})S(t) w}_{L^2},
\end{split}
\end{align}
a short computation (see \cite[Lem. 9.12]{hamster2017}) shows that
\begin{align}
\label{p4:eq:supb:ib:splitip}
\begin{split}
    \langle S(t) v,
    S(t) w \rangle_{H^1}
     = & \, \mathcal{J}_{\mathrm{tw}}(t)[v,w]
    -\frac{1}{2\rho} \frac{d}{d t} 
       \langle S(t) v,
    S(t) w \rangle_{L^2}
\end{split}
\end{align}
holds for all $t > 0$ 
and $v,w\in L^2$.
This identity allows the regularity issues that arise in
\S\ref{p4:sec:supb} and \S\ref{p4:sec:supbd} to be resolved.

\begin{lem}
\label{p4:lem:prml:semigroup}
Writing $\Pi = I - P_{\mathrm{tw}}$,
there exists a constant $M \geq 1$
so that
for every $t > 0$ we have the bounds
\begin{equation}
\label{p4:eq:prlm:semigroup:bnds:i}
\begin{array}{lcl}
\nrm{S(t)}_{\L(L^2,L^2)}&\leq& M,
  \\[0.2cm]
\nrm{S(t)\Pi}_{\L(L^2,L^2)}&\leq& M e^{-\b t},
  \\[0.2cm]
\nrm{S(t)\Pi}_{\L(L^2,H^1)}&\leq& M t^{-\frac{1}{2}}e^{-\b t}  ,
\\[0.2cm]
\norm{[\mathcal{L}_\mathrm{tw} - \rho \partial_{\xi\xi}] S(t)\Pi}_{\L(L^2,L^2) }
  & \leq & M t^{-\frac{1}{2}}e^{-\b t},
\\[0.2cm]
\norm{[\mathcal{L}^*_\mathrm{tw} - \rho \partial_{\xi\xi}] S(t)\Pi }_{\L(L^2,L^2) }
  & \leq & M t^{-\frac{1}{2}}e^{-\b t},
\\[0.2cm]
\norm{(S(t) - I)S(1)}_{\mathcal{L}(L^2,L^2)} & \leq & M \abs{t}.
\end{array}
\end{equation}
In particular, for any $t > 0$ and $v,w \in L^2$ we obtain the estimate
\begin{equation}
\label{p4:eq:supb:ib:Bbound}
    \abs{ \mathcal{J}_{\mathrm{tw}}(t)[\Pi v,\Pi w] }
    \leq M^2e^{-2\beta t}\left(1+\rho^{-1} t^{-1/2}\right)\norm{v}_{L^2} \norm{w}_{L^2}.
\end{equation}
\end{lem}
\begin{proof}
The bounds \sref{p4:eq:prlm:semigroup:bnds:i}
can be deduced from \cite[Prop. 5.2.1]{lorenzi2004analytic},
while \sref{p4:eq:supb:ib:Bbound} follows readily by inspecting
\sref{p4:eq:prlm:def:j:tw}.
\end{proof}

\subsection{Moment estimates and tail bounds} 
\label{p4:sec:prlm:mom:tail}

We briefly review here the technique that we use to pass back and forth between moment estimates and tail probabilities. The former are easier to estimate, but the latter are better suited for handling maxima. Our computations
are based heavily on \cite[Lem. 2.2.3]{talagrand}
and \cite{veraar2011note}.

\begin{lem}
\label{p4:lem:prml:veraar}
Consider a random variable $Z \geq 0$ and suppose that there exists a $\Theta>0$
so that the bound
\begin{equation}
\label{p4:eq:prml:moment}
    E [Z^{2p}] \leq p^p \Theta^{2p}
\end{equation}
holds for all integers $p  \geq 1$.
Then for every $\vartheta > 0$ we have the estimate
\begin{equation}
\label{p4:eq:prlm:tail:bnd}
    P( Z > \vartheta ) \le
    2 \, \mathrm{exp}\left[ - \frac{\vartheta^2}{2 e \Theta^2}\right] .
\end{equation}

\end{lem}
\begin{proof}
For any $\nu > 0$ a formal computation establishes the Chernoff bound
\begin{equation}
\begin{array}{lcl}
     P(Z>\vartheta)
     &=&P(e^{\nu Z^2}>e^{\nu \vartheta^2})
     \\[0.2cm]
     &\leq &
     e^{-\nu \vartheta^2}
     E \left[e^{\nu Z^2}\right]
    \\[0.2cm]
    & \leq & 
    e^{-\nu \vartheta^2}
       E\left[\sum_{p=0}^\infty\frac{\nu^p}{p!}Z^{2p}\right]
       \\[0.2cm]
     &\leq &e^{-\nu \vartheta^2}
     \sum_{p=0}^\infty \frac{\nu^p}{p!}
     p^p \Theta^{2p}  .
\end{array}
\end{equation}
Using $p! \geq p^{p} e^{-p}$
we obtain
\begin{equation}
\begin{array}{lcl}
     P( Z>\vartheta)
& \leq &
e^{-\nu \vartheta^2}
     \sum_{p=0}^\infty 
     \nu^p
     e^{p}\Theta^{2p}  ,
\\[0.2cm]
\end{array}
\end{equation}
which leads to  \sref{p4:eq:prlm:tail:bnd}
by choosing $\nu = (2e \Theta^2)^{-1}$.
\end{proof}

\begin{lem}
\label{p4:lem:prml:Tal}
Fix two constants $A \ge 2$ and $\Theta > 0$ and consider a random variable $Z \ge 0$ that satisfies the estimate
\begin{equation}
    P(Z > \vartheta) \le
    2 A \mathrm{exp}\big[ - \frac{\vartheta^2}{2e \Theta^2} \big]
\end{equation}
for all $\vartheta > 0$. Then for any $p \ge 1$
we have the moment bound
\begin{equation}
    E [ Z^{2p} ]
    \le \Big(p^p + \ln(A)^p \big) (8e\Theta^{2})^p
\end{equation}
\end{lem}
\begin{proof}
We pick an arbitrary $u_0$ and compute
\begin{equation}
\begin{array}{lcl}
E [ Z^{2p} ] & = & \int_0^\infty P( Z^{2p} > u) \, d u
\\[0.3cm]
& = & \int_0^\infty P( Z > \sqrt[2p]{u} ) \, d u 
\\[0.3cm]
& = & \int_0^{u_0} P (Z > \sqrt[2p]{u}) \, d u
 + \int_{u_0}^{\infty} P (Z > \sqrt[2p]{u}) \, d u
\\[0.3cm]
& \le &  u_0 + 2 A \int_{u_0}^\infty  e^{-u^{1/p}/(2e\Theta^2)} \, d u .
\\[0.3cm]
\end{array}
\end{equation}
Writing $w_0 = u_0^{1/p}/(2e\Theta^2)$
and recalling the upper incomplete
gamma function $\Gamma$,
we obtain
\begin{equation}
\begin{array}{lcl}
E [ Z^{2p} ]
& = & u_0
+ 2 p A (2e\Theta^2)^p \int_{w_0}^\infty
  v^{p-1} e^{-v} \, dv
\\[0.2cm]
& = & 
u_0 + 2 p A (2e\Theta^2)^p \Gamma(p , w_0) .
\end{array}
\end{equation}
Upon fixing $w_0 = 2p +\ln A$,
we may use \cite[eq. (1.5)]{borwein2009uniform}
for $p >1$ and a direct computation for $p=1$
to conclude that
\begin{equation}
    \Gamma(p, w_0) \le 2 w_0^{p-1} e^{-w_0}
\end{equation}
and hence
\begin{equation}
\begin{array}{lcl}
    E [ Z^{2p} ]
    & \le & (2e\Theta^2 w_0)^p
    + 4 p  A (2e \Theta^{2})^{p} w_0^{p-1} e^{-2p} A^{-1}
\\[0.2cm]
 & \le &  2 (2e\Theta^2 w_0)^p
 \\[0.2cm]
 & = &
 2 \big(2p +\ln( A) \big)^p (2e\Theta^{2})^p
 \\[0.2cm]
    & \le & 2^p \big( (2p)^p + \ln( A)^p \big) (2e\Theta^{2})^p,
\end{array}
\end{equation}
from which the desired bound follows.
\end{proof}

By applying a crude bound for tail-probabilities,
Lemmas \ref{p4:lem:prml:veraar} and \ref{p4:lem:prml:Tal} 
can be combined to control maximum expectations. This results in the following useful logarithmic growth estimate.
\begin{cor}\label{p4:cor:prml:max}
Consider $N\geq 2$ non-negative random variables $Y_1,Y_2,...,Y_N$ and suppose that there exists $\Theta > 0 $ so that the bound
\begin{align}
    E\left[Y_i^{2p}\right]\leq p^p \Theta^{2p}
\end{align}
holds for all integers $p \geq 1$ and each  $i\in\{1,..,N\}$.
Then for any $p \ge 1$ we have the 
bound
\begin{equation}
\begin{array}{lcl}
    E\max_{i \in \{1, \ldots, N\}} Y^{2p}_i
    &\leq&
     \Big(p^p + \ln(N)^p \big) (8e \Theta^{2})^p.
\end{array}
\end{equation}
\end{cor}
\begin{proof}
For any $\vartheta>0$ we may use Lemma \ref{p4:lem:prml:veraar} to estimate
\begin{align}
    P\big(\max_{i \in \{1, \ldots, N \}}Y_i>\vartheta \big) \leq \sum_{i=1}^N P(Y_i>\vartheta)\leq 2N\exp\left(-\frac{\vartheta^2}{2 e \Theta^2}\right),
\end{align}
so we can directly apply Lemma \ref{p4:lem:prml:Tal}. 
\end{proof}

\subsection{Supremum bounds}
\label{p4:sec:prlm:sup}
In this subsection we collect several key results that we will use to understand
stochastic convolutions such as \sref{p4:eq:int:exL2}. In order to setup such integrals in a precise fashion, we follow the extensive
discussion in \cite[\S5]{hamster2020} and introduce the Hilbert space
\begin{equation}
    L^2_Q = L^2_Q(\R) = Q^{1/2}\left(L^2(\R) \right) ,
\end{equation}
together with the set of Hilbert-Schmidt operators
\begin{align}
    HS = HS(L^2_Q,L^2) = HS\big(L^2_Q(\R),L^2(\R)\big)
\end{align}
that map $L^2_Q(\R)$ into $L^2(\R)$. Choosing an orthonormal basis $(e_k)$ for $L^2(\R)$, we recall that the Hilbert-Schmidt norm of the operator $B$ is given by
\begin{equation}
    \norm{B}_{HS}^2 = \sum_{k=0}^\infty \norm{B \sqrt Q{e}_k}_{L^2}^2.
\end{equation}
Fixing a complete filtered probability space
$\Big(\Omega, \mathcal{F}, ( \mathcal{F}_t)_{t \ge 0} , \mathbb{P} \Big)$,
it turns out 
\cite{concise,revuz2013continuous,Karczewska2005} that stochastic integrals against $dW^Q_t$ are well-defined if the integrand is taken from the class
\begin{equation}
\begin{array}{lcl}
    \mathcal{N}^2([0,T];(\mathcal{F}_t),HS) &:=&\{ B \in L^2\big( [0 , T] \times \Omega ;
  dt \otimes \mathbb{P} ; 
   HS \big):
  \\[0.2cm]
  & & \qquad \qquad 
  B \text{\, has a progressively \,} 
    (\mathcal{F}_t)\hbox{-measurable version}\}.
\end{array}
\end{equation}

Our previous results in \cite{hamster2017,hamster2018uneq,hamster2020}
relied heavily on various versions of the Burkholder-Davis-Gundy inequality,
but we only used the special case $p = 1$. The general form is stated
below, where we highlight the $p$-dependence of the prefactors
on the right-hand sides. We emphasize that
\sref{p4:eq:prlm:bdg:mild} - due to Veraar and Weis - is much more delicate
than standard martingale inequalities on account
of the convolution, which requires the semigroup to have special regularity properties.

\begin{lem}
\label{p4:lem:prml:BDG} 
There exists\footnote{Let us emphasize that all constants
that appear in this paper do not depend on $T$.} 
a constant $K_{\mathrm{cnv}} \geq 1$
so that for any $T > 0$, any integer $p \ge 1$  and any integrand
$B \in \mathcal{N}^2\left( [0,T] ; (\mathcal{F}_t ) ;
   HS(L_Q^2,L^2)\right)$ we have the bound
\begin{align}
\label{p4:eq:prlm:bdg}
E \sup_{0 \leq t \leq T}
\norm{
 \int_0^t B(s) \, d W^Q_s
}_{L^2}^{2p}
\leq &\,
K_{\mathrm{cnv}}^{2p} p^p
  E \left[ 
    \int_0^T \norm{B(s)}_{HS}^2 \, ds
  \right]^p ,
\intertext{together with its mild counterpart}
\label{p4:eq:prlm:bdg:mild}
E \sup_{0 \leq t \leq T}
\norm{
 \int_0^t S(t-s)B(s) \, d W^Q_s
}_{L^2}^{2p}
\leq&\,
K_{\mathrm{cnv}}^{2p} p^p
  E \left[ 
    \int_0^T \norm{B(s)}_{HS}^2 \, ds
  \right]^p .
\end{align}
\end{lem}
\begin{proof}
We note first that $L^2(\R)$ is a Banach space of type 2. In particular,
\sref{p4:eq:prlm:bdg} follows from 
\cite[Prop. 2.1 and Rem. 2.2]{veraar2011note};
see also \cite[Thm. 4.36]{dapratozab}.
In addition, the linear operator $\mathcal{L}_{\mathrm{tw}}$ admits a bounded $H^\infty$-calculus \cite[Lem. 9.7]{hamster2017},
which allows us to apply \cite[Thm. 1.1]{veraar2011note}
and obtain \sref{p4:eq:prlm:bdg:mild}.
\end{proof}

We remark that the inequalities \sref{p4:eq:prlm:bdg}-\sref{p4:eq:prlm:bdg:mild}
are very strong and useful on short time intervals, but on longer timescales
it is no longer possible to exploit the decay properties of the semigroup.
Indeed, the right-hand side of \sref{p4:eq:prlm:bdg:mild} grows linearly in time
for integrands that are constant - as for the Ornstein-Uhlenbeck processes.
This changes if one drops the supremum.

\begin{cor}
\label{p4:cor:prml:BDG}
Consider the setting of Lemma \ref{p4:lem:prml:BDG}. 
Then for any $0 \leq t \leq T$ and any integer $p \ge 1$ we have the bound
\begin{equation}
\label{p4:eq:prlm:bdg:naked}
E \norm{
 \int_0^t S(t-s)B(s) \, d W^Q_s
}_{L^2}^{2p}
\leq 
K_{\mathrm{cnv}}^{2p} p^p
  E \left[ 
    \int_0^t \norm{S(t-s)B(s)}_{HS}^2 \, ds
  \right]^p .
\end{equation}
\end{cor}
\begin{proof}
Note that 
\begin{equation}
E \norm{
 \int_0^t S(t-s)B(s) \, d W^Q_s
}_{L^2}^{2p}
\leq  E\sup_{0\leq \tilde t\leq t}\norm{
 \int_0^{\tilde t} S(t-s)B(s) \, d W^Q_s
}_{L^2}^{2p},
\end{equation}
so the result follows directly from \sref{p4:eq:prlm:bdg}. 
\end{proof}

The following general result due to
Dirksen generalizes \cite[eq. (2.49)]{talagrand}
and is the key ingredient that will allow us to significantly
improve the bound \sref{p4:eq:prlm:bdg:mild}. It is based
on the chaining principle developed by Talagrand, which requires us to understand
the tail behavior of the probability distribution for
the temporal increments of stochastic process.

\begin{thm}[\cite{dirksen2015tail}]
\label{p4:thm:supb:talagrand}
There exists a constant $C_\mathrm{ch} \ge 1$ so that the following holds true.
Consider a stochastic process 
$X: [0,T] \to L^2$ for some $T > 0$
with paths that are almost-surely continuous. Suppose furthermore
that there exists a metric $d = d(\cdot, \cdot)$ on $[0,T]$
so that the increments of $X$ satisfy the estimate
\begin{equation}
\label{p4:eq:prml:incrc}
P\left( 
\norm{X(t_1) - X(t_2) }_{L^2} > \vartheta \right)\leq 2 \exp\left(-\frac{\vartheta^2}{2d(t_1,t_2)^2}\right),
\end{equation}
for every $t_1, t_2 \in [0,T]$ and $\vartheta > 0$.
Then for any integer $p \ge 1$ we have the bound
\begin{align}
    E\sup_{0\leq t\leq T }\nrm{X(t)}_{L^2}^{2p}\leq C_\mathrm{ch}^{2p} \left(\int_0^\infty \sqrt{\ln\big(N(T,d,\nu)\big)}d\nu\right)^{2p}    +  C_\mathrm{ch}^{2p} p^p
         \mathrm{diam}(T,d)^{2p},
\end{align}
where $N(T,d,\nu)$ is the smallest number of intervals of length at most $\nu$
in the metric $d$ required to cover $[0,T]$
and $\mathrm{diam}(T,d)$ is the diameter
of $[0,T]$ in this metric.
\end{thm}
\begin{proof}
This bound follows by choosing
$\alpha = 2$ in \cite[eq. (3.2)]{dirksen2015tail}
and applying the final
inequality in the proof of \cite[Thm. 3.2]{dirksen2015tail}.
\end{proof}

\section{Supremum bounds for stochastic integrals}
\label{p4:sec:supb}
In this section we develop the machinery needed to 
obtain bounds for two types of stochastic integrals. 
In particular,  we introduce the $L^2$-valued integral
\begin{equation}
    \eb(t)=\int_0^tS(t-s)B(s)dW^Q_s
\end{equation}
together with the scalar integral
\begin{align}
\label{eq:supb:sc:itg}
\mathcal{I}^{\mspace{1mu}\mathrm{s}}_B(t)
 = &
\int_0^t e^{-\e(t-s)}
\int_0^{s}\ip{S(s-s') V(s'),S(s-s')B(s') \, dW^Q_{s'} }_{H^1} \, ds
\end{align}
and set out to obtain bounds for the quantities
\begin{equation}
    E \sup_{0 \leq t \leq T} \norm{\eb(t)}_{L^2}^{2p},
    \qquad 
    \qquad
    E \max_{i \in \{1 , \ldots, T\}} \abs{\mathcal{I}^{\mspace{1mu}\mathrm{s}}_B(i) }^{2p}.
\end{equation}
Recalling the constant $\beta > 0$ introduced
in Lemma \ref{p4:lem:prml:semigroup}, we 
take $\varepsilon \in (0, \beta)$
for the parameter 
appearing in \sref{eq:supb:sc:itg},
which we consider to be fixed throughout the entire section.

We will use the first of these expressions in \S\ref{p4:sec:nls:L2} to control
the $L^2$-norm of $V(t)$, while the second term
plays a crucial role in \S\ref{p4:sec:nls:H1} where we bound 
the $H^1$-norm of $V(t)$ in an integrated sense. In both cases
$B$ will be replaced by a (complicated) function of $V$, but
we make use of a generic placeholder here in order to emphasize
the broad applicability of our techniques. Indeed, we only need
to impose the following two general conditions on our integrands.

\begin{itemize}
    \item[(hB)]{
      The process $B \in \mathcal{N}^2\left( [0,T] ; (\mathcal{F}_t ); HS(L^2_Q,L^2) \right)$ satisfies
      \begin{align}
        \label{p4:eq:supb:b:psi:tw}
         \ip{B(t)v,\psi_{\mathrm{tw}}}_{L^2}=0
      \end{align}
      for all $t \in [0,T]$ and $v \in L^2_Q$.
      Furthermore, there exists
      $\Theta_* > 0$ so that the following pathwise bounds hold for all $0 \leq t \leq T$: 
      \begin{align}
      \label{p4:eq:supb:assumB}
      \begin{split}
      \int_0^t e^{-\e(t-s)} \norm{B(s)}_{HS}^{2}ds
    \leq \Theta_*^{2},\hspace{1cm}
        \norm{S(1) B(t)}_{HS}^{2}
                \leq \Theta_*^2.
    \end{split}
    \end{align}
    }
    \item[(hV)]   The process $V \in \mathcal{N}^2\left( [0,T] ; (\mathcal{F}_t ); H^1 \right)$ 
    satisfies
      \begin{align}
        \label{p4:eq:st:ip:v:psi:tw}
         \ip{V(t),\psi_{\mathrm{tw}}}_{L^2}=0
      \end{align}
      for all $t \in [0, T]$.
      Furthermore, there exists a $\Lambda_* > 0$ so that the pathwise bound
      \begin{equation}
      \nrm{V(t)}_{L^2}\leq \Lambda_*
      \end{equation}
      holds for all $0\leq t \leq T$. 
\end{itemize}

We remark that \sref{p4:eq:supb:b:psi:tw}
and \sref{p4:eq:st:ip:v:psi:tw} imply that $B$ and $V$ do not
feel the neutral mode of the semigroup. This allows us to use the decay rates
from Lemma \ref{p4:lem:prml:semigroup} and establish our main result below.
In particular, we obtain \sref{p4:eq:supb:mr:eb} in \S\ref{p4:sec:supb:eb}
and \sref{p4:eq:supb:mr:ib} in  \S\ref{p4:sec:supb:ib}. For convenience, we will consider $T$ to be an integer from now on. This will make the splitting of the integrals easier in the following sections and any results for non-integer $T$ can be established by rounding $T$ up to the nearest integer.

\begin{prop}
\label{p4:prop:supb:mr}
There exists a constant $K \ge 1$ so that
for any integer $T \geq 2$,
any pair of processes $(B,V)$ that satisfies (hB) and (hV) and any $p\ge 1$,
we have the supremum bound 
\begin{align}
\label{p4:eq:supb:mr:eb}
E \sup_{0 \leq t \leq T} \norm{\mathcal{E}_{B}(t)}_{L^2}^{2p} &\leq
\big( p^p + \ln(T)^p  \big) K^{2p} \Theta_*^{2p} 
\intertext{together with its counterpart}
\label{p4:eq:supb:mr:ib}
E \max_{i\in \{1,....,T\}}\abs{\mathcal{I}^{\mspace{1mu}\mathrm{s}}_{B}(i)}^{p} 
  &\leq\big( p^{p/2} + \ln(T)^{p/2}  \big)  K^{p} \Lambda_*^{p} \Theta_*^{p}
\end{align}
\end{prop}

\subsection{Estimates for $\mathcal{E}_B$}
\label{p4:sec:supb:eb}
Motivated by the considerations in the introduction,
we make the splitting
\begin{equation}
  \mathcal{E}_{B}(t) = \mathcal{E}_{B}^{\mathrm{lt}}(t) + \mathcal{E}_{B}^{\mathrm{sh}}(t),
\end{equation}
in which the short time (sh) and long time (lt) contributions 
are respectively given by 
\begin{equation}
    \mathcal{E}_{B}^{\mathrm{lt}}(t)
    = \int_0^{t-1} S(t-s) B(s) \, d W^Q_s,
    \qquad \qquad
    \mathcal{E}_{B}^{\mathrm{sh}}(t)
    = \int_{t-1}^{t} S(t-s) B(s) \, d W^Q_s ,
\end{equation}
where we interpret the boundary $t-1$ as $\max\{t-1,0\}$ if necessary. 
Both these terms need to be handled using separate techniques.

\paragraph{Short time bounds}
Remembering that $T$ is an integer, we introduce the function
\begin{equation}
\begin{array}{lcl}
    \Upsilon_{B}^{(i)}
    & = & \sup_{0 \leq s \leq 1} \norm{ \int_i^{i+s} S(i+s - s') B(s') \, d W^Q_{s'} }_{L^2}
    \end{array}
\end{equation}
for any $i \in \{0, \ldots, T-1\}$, while for $i=-1$ we define $\Upsilon_{B}^{(i)}=0$.  An elementary computation
allows us to bound $\mathcal{E}^{\mathrm{sh}}_B(t)$ in terms of 
at most two of this finite set of quantities.

\begin{lem}
\label{p4:lem:supb:st:split}
Pick any integer $T \geq 2$ and assume that (hB) holds.
Then for all $0 \leq t \leq T$ we have the bound
\begin{equation}
    \norm{\mathcal{E}_B^{\mathrm{sh}}(t)}_{L^2}
    \leq 2 M \Upsilon_{B}^{(\lfloor t \rfloor - 1)}
      + \Upsilon_{B}^{(\lfloor t \rfloor )}.
\end{equation}
\end{lem}
\begin{proof}
Since the estimate is immediate for $0 \leq t < 1$, we pick $t \geq 1$. Splitting
the integral yields
\begin{align}
\begin{split}
    \norm{\mathcal{E}_B^{\mathrm{sh}}(t)}_{L^2} 
      \leq&\, \norm{ \int_{t-1}^{\lfloor t \rfloor} S(t-s)B(s) dW^Q_s}_{L^2}
    + \norm{ \int_{\lfloor t \rfloor }^{ t} S(t-s)B(s) dW^Q_s}_{L^2}
    \\[0.2cm]
     \leq &\,
    \norm{ \int_{\lfloor t \rfloor-1}^{\lfloor t \rfloor} S(t-s)B(s) dW^Q_s}_{L^2}
    + \norm{ \int_{\lfloor t \rfloor-1}^{t-1} S(t-s)B(s) dW^Q_s}_{L^2}\\[0.2cm]
    &+ \norm{ \int_{\lfloor t \rfloor }^{ t} S(t-s)B(s) dW^Q_s}_{L^2}.
\end{split}
\end{align}
Using Lemma \ref{p4:lem:prml:semigroup} we obtain the estimate
\begin{equation}
\begin{array}{lcl}
\norm{ \int_{\lfloor t \rfloor-1}^{\lfloor t \rfloor} S(t-s)B(s) dW^Q_s}_{L^2}  
&\le& \norm{S(t-\lfloor t \rfloor)}_{\mathcal{L}(L^2,L^2)}
\norm{ \int_{\lfloor t \rfloor-1}^{\lfloor t \rfloor} S(\lfloor t \rfloor-s)B(s) dW^Q_s}_{L^2}  
\\[0.2cm]
& \leq & 
M \norm{ \int_{\lfloor t \rfloor-1}^{\lfloor t \rfloor} S(\lfloor t \rfloor-s)B(s) dW^Q_s}_{L^2}  
\\[0.4cm]
& \leq &
M \Upsilon_{B}^{(\lfloor t \rfloor - 1)} ,
\end{array}
\end{equation}
together with
\begin{equation}
\begin{array}{lcl}
\norm{ \int_{\lfloor t \rfloor-1}^{t-1} S(t-s)B(s) dW^Q_s}_{L^2}
&\le& \norm{S(1)}_{\mathcal{L}(L^2,L^2)}
\norm{ \int_{\lfloor t \rfloor-1}^{t-1} S(t-1-s)B(s) dW^Q_s}_{L^2}  
\\[0.2cm]
& \leq & 
M \Upsilon_{B}^{(\lfloor t \rfloor - 1)},
\end{array}
\end{equation}
from which the desired bound readily follows.
\end{proof}

\begin{cor}
\label{p4:cor:supb:3m}
Pick any integer $T\geq 2$
and assume that (hB) holds. Then  for all $p\geq 1$ we have the pathwise bound
\begin{equation}
\label{p4:eq:st:bnd:sup:eb:sh:vs:max}
    \sup_{0 \leq t \leq T} \norm{\mathcal{E}_B^{\mathrm{sh}}(t)}^{2p}_{L^2}
    \leq \left(3M\right)^{2p} \max_{i \in \{0, \ldots, T-1\}}  \left(\Upsilon_{B}^{(i)}\right)^{2p}.
\end{equation}
\end{cor}

The expectation of the right-hand  side of 
\sref{p4:eq:st:bnd:sup:eb:sh:vs:max} can be controlled
using Corollary \ref{p4:cor:prml:max}. We hence require moment bounds on $\Upsilon_B^{(i)}$,
which can be obtained by applying the mild Burkholder-Davis-Gundy inequality.
Here we use the crucial fact that $\mathcal{L}_{\mathrm{tw}}$
admits an $H^\infty$-calculus.

\begin{lem}
\label{p4:lem:supb:st:moment}
Pick any integer $T\geq 2$
and assume that (hB) holds. Then for any integer $p \geq 1$
and any $i \in \{0, \ldots , T-1\}$ we have the bound
\begin{equation}
\label{p4:eq:st:bnd:sup:eb:ups:bnd}
  E \left[\Upsilon_B^{(i)}\right]^{2p}
     \leq   K_{\mathrm{cnv}}^{2p} p^{p}
     e^{\e p} \Theta_*^{2p}.
\end{equation}
\end{lem}
\begin{proof}
Applying Lemma \ref{p4:lem:prml:BDG}, we readily compute  
\begin{equation}
\begin{array}{lcl}
    E \left[\Upsilon_B^{(i)}\right]^{2p}
     &\leq& K_{\mathrm{cnv}}^{2p} p^p
     E \left[ \int_i^{i+1}\norm{B(s)}_{HS}^2  \, ds \right]^p
\\[0.2cm]
& \leq & 
      K_{\mathrm{cnv}}^{2p} p^p e^{\e p}
     E \left[ \int_0^{i+1}e^{-\e(i+1 - s)}\norm{B(s)}_{HS}^2  \, ds \right]^p ,
\end{array}
\end{equation}
which implies the stated bound on account of \sref{p4:eq:supb:assumB}. 
\end{proof}

\paragraph{Long-term bounds}
The goal here is to apply the chaining result from
Theorem \ref{p4:thm:supb:talagrand} to the long-term integral
$\eb^{\mathrm{lt}}$. To achieve this,
we will use Lemma \ref{p4:lem:prml:veraar} to turn
moment bounds for the increments
of $\eb^{\mathrm{lt}}$ into the desired
tail bounds for the associated probability distribution.

For any pair $0 \leq t_1 \leq t_2 \leq T$, we split this increment
into two parts
\begin{equation}
    \mathcal{E}_B^{\mathrm{lt}}(t_1) - \mathcal{E}_B^{\mathrm{lt}}(t_2)
    = 
    \mathcal{I}_1(t_1, t_2)
    + \mathcal{I}_2(t_1, t_2)
\end{equation}
that are defined by
\begin{equation}
\begin{array}{lcl}
    \mathcal{I}_1(t_1, t_2)
    & = & \int_0^{t_1 - 1} 
    \left[S(t_2 - s) - S(t_1 - s)\right] B(s) \, d W^Q_{s},
\\[0.3cm]
\mathcal{I}_2(t_1, t_2)
    & = & \int_{t_1 - 1}^{t_2 - 1} S(t_2 - s) B(s) \, d W^Q_{s}.
\end{array}
\end{equation}
The first of these can be analyzed by exploiting the regularity of the
semigroup $S(t-s)$ for $t-s\geq 1$, while the second 
requires a supremum bound on the `smoothened' process  $S(1)B$, hence explaining the assumption in equation \sref{p4:eq:supb:assumB}.

\begin{lem}
\label{p4:eq:st:eb:lt:i}
Pick any integer $T \geq 2$ and assume that (hB) holds.
Then for any $1 \leq t_1 \leq t_2 \leq T$ 
and any integer $p \geq 1$ we have the bound
\begin{equation}
E \norm{ \mathcal{I}_1(t_1, t_2)}_{L^2}^{2p}
\leq p^{p} K_{\mathrm{cnv}}^{2p} M^{4p} \Theta_*^{2p} \abs{t_2 - t_1}^{2p}.
\end{equation}
\end{lem}
\begin{proof}
Observe first that
\begin{equation}
\begin{array}{lcl}
E \norm{ \mathcal{I}_1(t_1, t_2)}_{L^2}^{2p}
& \leq &
\norm{ [S(t_2 - t_1) - I]S(1) }_{\mathcal{L}(L^2,L^2)}^{2p}
E
\norm{
  \int_0^{t_1-1}
  S(t_1 -1 - s) B(s) \, d W^Q_{s}
}_{L^2}^{2p}
\\[0.2cm]
& \leq & 
M^{2p} \abs{t_2 - t_1}^{2p}
E
\norm{
  \int_0^{t_1-1}
  S(t_1 -1 - s) B(s) \, d W^Q_{s}
}_{L^2}^{2p}.
\end{array}
\end{equation}
Applying \sref{p4:eq:prlm:bdg:naked}
with $T = t_1 - 1$, we find
\begin{equation}
\begin{array}{lcl}
E \norm{ \mathcal{I}_1(t_1, t_2)}_{L^2}^{2p}
& \leq &
p^p K_{\mathrm{cnv}}^{2p}
M^{2p} \abs{t_2 - t_1}^{2p}
E \left[
  \int_0^{t_1-1}
  \norm{
  S(t_1 -1 - s) B(s)
  }_{HS}^2 \, ds\right]^p
\\[0.4cm]
& \leq &
p^p K_{\mathrm{cnv}}^{2p}
M^{4p} \abs{t_2 - t_1}^{2p}
E \left[
  \int_0^{t_1-1}
  e^{-2 \beta(t_1 - 1 - s)}
  \norm{
    B(s)
  }_{HS}^2 \, ds \right]^p,
\end{array}
\end{equation}
which yields the stated bound in view of \sref{p4:eq:supb:assumB}.
\end{proof}

\begin{lem}
\label{p4:eq:st:eb:lt:ii}
Pick any integer $T \geq 2$ and assume that (hB) holds.
Then for any $1 \leq t_1 \leq t_2 \leq T$ 
and any integer $p \geq 1$ we have the bound
\begin{equation}
E \norm{ \mathcal{I}_2(t_1, t_2)}_{L^2}^{2p}
\leq p^p K_{\mathrm{cnv}}^{2p} M^{2p} \Theta_*^{2p} \abs{t_2 - t_1}^{p}.
\end{equation}
\end{lem}
\begin{proof}
It suffices to compute
\begin{equation}
\begin{array}{lcl}
E \norm{ \mathcal{I}_2(t_1, t_2)}_{L^2}^{2p}
& = &
E \left[
\norm{
  \int_{t_1-1}^{t_2 -1}
  S(t_2 -1 - s) S(1) B(s)
  \, dW^Q_{s}
  }_{L^2}\right]^{2p}
\\[0.3cm]
& \leq &
p^p K_{\mathrm{cnv}}^{2p}
E \left[
  \int_{t_1-1}^{t_2-1}
  \norm{S(t_2 -1 - s)}_{\mathcal{L}(L^2,L^2)}^2
  \norm{S(1)B(s)}_{HS}^2 \, ds\right]^p
\\[0.3cm]
& \leq &
p^p K_{\mathrm{cnv}}^{2p}
M^{2p} \abs{t_2 - t_1}^{p} E\left[
  \sup_{t_1 - 1 \leq s \leq t_2 - 1} \norm{S(1) B(s)}_{HS}^2\right]^p 
\end{array}
\end{equation}
and apply \sref{p4:eq:supb:assumB}.
\end{proof}

The previous two results were tailored to handle small
increments $\abs{t_2 - t_1} \leq 1$. For larger increments
one can exploit the decay of the semigroup to show that
$\mathcal{E}^{\mathrm{lt}}_B$ remains bounded in expectation.

\begin{lem} Pick any integer $T \geq 2$ and assume that (hB) holds.
\label{p4:eq:st:eb:lt:glb}
Then for any $0 \leq t \leq T$ and any integer $p \geq 1$ we have the bound
\begin{equation}
    E 
    \norm{\mathcal{E}_B^{\mathrm{lt}}(t)}_{L^2}^{2p}
    \le
    p^p K_{\mathrm{cnv}}^{2p}
     M^{2p} \Theta_*^{2p}.
\end{equation}
\end{lem}
\begin{proof}
Using Corollary \ref{p4:cor:prml:BDG}, we find
\begin{equation}
\begin{array}{lcl}
E \norm{\mathcal{E}_B^{\mathrm{lt}}(t)}_{L^2}^{2p}
& \leq &
p^p K_{\mathrm{cnv}}^{2p}
E \left[
  \int_0^{t-1}
  \norm{S(t -s)\Pi}_{\mathcal{L}(L^2,L^2)}^2
  \norm{B(s)}_{HS}^2 \, ds\right]^p
\\[0.4cm]
& \leq &
p^p K_{\mathrm{cnv}}^{2p}
M^{2p} 
E \left[ 
  \int_0^{t-1}
  e^{-2 \beta(t - 1 - s)}  
  \norm{ B(s)}_{HS}^2ds
  \right]^p
\\[0.4cm]
& \leq &
p^p K_{\mathrm{cnv}}^{2p}
M^{2p} \Theta_*^{2p}.
\end{array}
\end{equation}
\end{proof}

\begin{cor}
\label{p4:cor:st:eb:lt:full:increment}
Pick any integer $T \geq 2$ and assume that (hB) holds.
Then for any $0 \leq t_1 \leq t_2 \leq T$ and any integer $p \geq 1$
we have the bound
\begin{equation}
    E 
    \norm{\mathcal{E}_B^{\mathrm{lt}}(t_1)
    -
    \mathcal{E}_B^{\mathrm{lt}}(t_2)
    }_{L^2}^{2p}
    \le
     2^{2p} p^p K_{\mathrm{cnv}}^{2p}M^{4p} \Theta_*^{2p}
     \min\{ \abs{t_2-t_1}^{1/2} , 1 \}^{2p}.
\end{equation}
\end{cor}
\begin{proof}
This follows from the standard inequality
$(a+ b)^{2p} \leq 2^{2p - 1}( a^{2p} + b^{2p} )$ and
a combination of the estimates from Lemmas
\ref{p4:eq:st:eb:lt:i}-\ref{p4:eq:st:eb:lt:glb}.
\end{proof}

\begin{lem}
There exists a constant
$K_\mathrm{lt}\geq 1$ 
so that for any integer $T \geq 2$,
any process $B$ that satisfies (hB)
and any $p \ge 1$,
we have the supremum bound 
\begin{align}
\label{p4:lem:st:eb:lt:chn:bnd}
    E\sup_{0\leq t\leq T} \nrm{\mathcal{E}_B^\mathrm{lt}(t)}_{L^2}^{2p} \leq \big( \ln(T)^p + p^p\big) K_{\mathrm{lt}}^{2p} \Theta_*^{2p} .
\end{align}
\end{lem}
\begin{proof}
Upon writing $d_{\max} = 2 \sqrt{e} K_{\mathrm{cnv}} M^2 \Theta_*$
together with
\begin{equation}
    d(t_1, t_2)
    = d_{\max} \min\{ \sqrt{|t_2 - t_1|} , 1 \},
\end{equation}
an application of Lemma \ref{p4:lem:prml:veraar} to Corollary \ref{p4:cor:st:eb:lt:full:increment} provides the bound
\begin{equation}
\begin{array}{lcl}
    P\left( \norm{\mathcal{E}_B^{\mathrm{lt}}(t_1)
    -
    \mathcal{E}_B^{\mathrm{lt}}(t_2)
    }_{L^2} > \vartheta \right)
& \leq & 
2\,\mathrm{exp}\left[
      - \frac{\vartheta^2}{ 2d(t_1,t_2)^2 }
    \right].
\end{array}
\end{equation}
Turning to the packing number $N(T,d,\nu)$
introduced in Theorem \ref{p4:thm:supb:talagrand},
we note that 
$N(T,d,\nu) = 1$
whenever $\nu \geq d_{\max}$,
while for 
smaller $\nu$ we have
\begin{equation}
    N(T, d, \nu)  \leq \frac{T d^2_{\max}}{\nu^2}.
\end{equation}
In particular, 
the Dudley entropy integral can be bounded by
\begin{align}\begin{split}
    \int_0^{\infty}
    \sqrt{\ln(N(T,d,\nu) ) } \, d \nu     \leq &
    \int_0^{d_{\max}}
    \sqrt{\ln (T d_{\max}^2 / \nu^2)  } \, d \nu\\[0.2cm]
    = &\int_0^{d_{\max}}\sqrt{-2\ln \big( \nu / ( d_{\max}\sqrt{T}) \big)}
    \, d \nu\\[0.2cm]
    =&\, d_\mathrm{max}\sqrt{T}\int_0^{1/\sqrt T}
    \sqrt{-2\ln(\nu) } \, d \nu\\[0.2cm]
    =&\,d_\mathrm{max}\left(\sqrt{2\ln(T) }+\sqrt{\pi}\sqrt{T}\text{erfc}\big(\sqrt{\ln(T) }\big)\right).
\end{split}
\end{align}
Since the function $\sqrt{T}\text{erfc}(\sqrt{\ln(T)})$ 
is uniformly bounded for $T \geq 2$,
the desired estimate
now follows directly from Theorem \ref{p4:thm:supb:talagrand}.
\end{proof}

\begin{proof}[Proof of \sref{p4:eq:supb:mr:eb}
in Proposition \ref{p4:prop:supb:mr}]
Applying Corollary \ref{p4:cor:prml:max}
to the estimates  \sref{p4:eq:st:bnd:sup:eb:sh:vs:max}-\sref{p4:eq:st:bnd:sup:eb:ups:bnd},
we directly find
\begin{align}
    E\sup_{0\leq t\leq T} \nrm{\eb^{\mathrm{sh}}(t)}^{2p}_{L^2}\leq
    \big(\ln(T)^p + p^p \big)\big(72 M^2 eK^2_\mathrm{cnv}e^{\e}\Theta^2_*\big)^p.
\end{align}
Combining this with the analogous long-term estimate \sref{p4:lem:st:eb:lt:chn:bnd}
readily yields the result.
\end{proof}

\subsection{Estimates for $\mathcal{I}^{\mspace{1mu}\mathrm{s}}_B$}
\label{p4:sec:supb:ib}

Our strategy here for controlling $\mathcal{I}^{\mspace{1mu}\mathrm{s}}_B$ is to appeal to Corollary \ref{p4:cor:prml:max},
which requires us to obtain moment bounds on $\mathcal{I}^{\mspace{1mu}\mathrm{s}}_B(i)$.
As a preparation, we switch the order of integration using a stochastic Fubini theorem \cite[Thm. 4.33]{dapratozab}, to find
\begin{equation}
\begin{array}{lcl}
    \mathcal{I}^{\mspace{1mu}\mathrm{s}}_B(i)
     &= & \int_0^i e^{-\e(i-s)}
    \int_0^{s}\ip{S(s-s') V(s'),S(s-s')B(s') \cdot  }_{H^1} \, dW^Q_{s'} \, ds 
\\[0.3cm]
    &= & \int_0^{i} \int_{s'}^{i}
    e^{-\e(i-s)}
    \ip{S(s-s') V(s'),S(s-s')B(s') \cdot  }_{H^1} \,ds\, dW^Q_{s'} .
\end{array}
\end{equation}
Corollary \ref{p4:cor:prml:BDG} hence yields
\begin{align}
\begin{split}
\label{p4:eq:st:bnd:i:for:i:b:s}
    E [\mathcal{I}^{\mspace{1mu}\mathrm{s}}_B(i)]^{2p}
    &\leq p^pK_{\mathrm{cnv}}^{2p}E\left[\int_0^{i}\sum_{k=0}^\infty
       \mathcal{K}^{(i)}_k(s')^2 ds'\right]^{p},\\[0.2cm]
    \end{split}
\end{align}
in which we have introduced the expression\footnote{Note that this integral is an improper integral, as the integrand is not defined for the lower boundary $s=s'$. In \cite{hamster2017} we show how this problem can be circumvented by replacing $s$ by $s+\d$ and subsequently sending $\d\to0$.}
\begin{equation}
\label{p4:eq:st:def:cal:k:i}
    \mathcal{K}^{(i)}_{k}(s') = \int_{s'}^{i}e^{-\e(i-s)}\ip{S(s-s')V(s'),S(s-s')B(s')\sqrt Qe_k}_{H^1}ds.
\end{equation}
Motivated by \sref{p4:eq:supb:ib:splitip}, we split
$\mathcal{K}^{(i)}_k$ into the two parts 
\begin{equation}
\begin{array}{lcl}
    \mathcal{K}^{(i)}_{I;k}(s') &=&\int_{s'}^{i}e^{-\e(i-s)}\mathcal{J}_\mathrm{tw}[V(s'),B(s')\sqrt Qe_k]ds  ,
    \\[0.2cm]
    \mathcal{K}^{(i)}_{II;k}(s') &=&-\frac{1}{2\rho}\int_{s'}^{i}e^{-\e(i-s)}\frac{d}{ds}\ip{S(s-s')V(s'),S(s-s')B(s')\sqrt Qe_k}_{L^2}ds.
\end{array}
\end{equation}
Performing an integration by parts, the second of these integrals can be further decomposed
into the three terms
\begin{equation}
\begin{array}{lcl}
\mathcal{K}^{(i)}_{IIa;k}(s')
& = & -\frac{\e}{2\rho} \int_{s'}^{i} e^{-\e(i-s    )}\ip{S(s-s')V(s'),S(s-s')B(s')\sqrt Qe_k}_{L^2}ds ,
\\[0.4cm]
\mathcal{K}^{(i)}_{IIb;k}(s')
& = & -\frac{1}{2\rho}\ip{S(i-s')V(s'),S(i-s')B(s')\sqrt Qe_k}_{L^2} ,
\\[0.4cm]
\mathcal{K}^{(i)}_{IIc;k}(s')
& = & \frac{1}{2\rho}e^{-\e(i-s')}
    \ip{V(s'),B(s')\sqrt Qe_k}_{L^2}.
\\[0.4cm]
\end{array}
\end{equation}

\begin{lem}
\label{p4:lem:st:bnds:cal:k:sq}
There exists a constant $K_{\mathcal{K}} \ge 1$ 
so that for any integer $T \geq 2$,
any pair of processes $(B,V)$ that satisfies (hB) and (hV)
and
any $i \in \{1, \ldots, T\}$,  we have the bound
\begin{equation}
\begin{array}{lcl}  
\sum_{k}
\mathcal{K}^{(i)}_{\#;k}(s') ^2
& \leq & K_{\mathcal{K}} e^{-2 \e (i-s')} \nrm{V(s')}^2_{L^2}\nrm{B(s')}^2_{HS}
\end{array}
\end{equation}
for all $0 \leq s' \leq i$ and each
of the symbols $\# \in \{I, IIa, IIb, IIc\}$.
\end{lem}
\begin{proof}
Upon introducing the expression
\begin{equation}
\label{p4:eq:st:def:k:eps:beta}
K(\e, \beta) =
  e^{\e(i-s')}
  \int_{s'}^{i}e^{-\e(i-s)}e^{-2\beta(s-s')}\left(1+\rho^{-1}(s-s')^{-1/2}\right)ds
\end{equation}
we may exploit \sref{p4:eq:supb:ib:Bbound} to obtain the bound
\begin{equation}
\begin{array}{lcl}
    \sum_{k} \mathcal{K}^{(i)}_{I;k}(s') ^2 
    & \le&   M^2 e^{-2\e(i - s')} \sum_{k}
    \nrm{V(s')}^2_{L^2}\nrm{B(s')\sqrt Q e_k}^2_{L^2}
       K(\e, \beta)^2
\\[0.2cm]
& = & M^2 
    e^{-2\e(i - s')} \nrm{V(s')}^2_{L^2}\nrm{B(s')}^2_{HS}
      K(\e, \beta)^2.
\end{array}    
\end{equation}
The estimate for $\# = I$ hence follows from the computation
\begin{equation}
\begin{array}{lclcl}
    K(\e, \beta)
    &\leq & 
      \int_{0}^{\infty}e^{(\e-2\beta)s}\left(1+\rho^{-1}s^{-1/2}\right)ds
    &= &\frac{1}{2\beta-\e}+\frac{1}{\rho}\sqrt{\frac{\pi}{2\b-\e}}.
\end{array}    
\end{equation}
The estimate for $\mathcal{K}^{(i)}_{IIa;k}$ can be obtained in the same fashion,
but here the $(s-s')^{-1/2}$ term in \sref{p4:eq:st:def:k:eps:beta}
is not required. Finally, the estimates for  $\mathcal{K}^{(i)}_{IIb;k}$
and $\mathcal{K}^{(i)}_{IIc;k}$ are immediate from Lemma \ref{p4:lem:prml:semigroup}
and the choice $\beta > \e$.
\end{proof}

\begin{proof}[Proof of \sref{p4:eq:supb:mr:ib} in Proposition \ref{p4:prop:supb:mr}]
Applying Young's inequality to the decomposition above,
we obtain the pathwise bound
\begin{equation}
\begin{array}{lclcl}
\int_0^{i}\sum_{k}
       \mathcal{K}^{(i)}_k(s')^2 ds'
& \leq &  16 K_\mathcal{K} \int_0^i e^{-2\e(i-s')}
\nrm{V(s')}^2_{L^2}\nrm{B(s')}^2_{HS}\, ds'
& \leq & 
 16 K_\mathcal{K} \Lambda_*^2 \Theta_*^2.
\end{array}
\end{equation}
In view of \sref{p4:eq:st:bnd:i:for:i:b:s} this implies
\begin{equation}
\begin{array}{lcl}
    E [\mathcal{I}^{\mspace{1mu}\mathrm{s}}_B(i)]^{2p}
    &\leq & 2^{4p}  p^p K_\mathcal{K}^p  K_{\mathrm{cnv}}^{2p}
    \Lambda_*^{2p} \Theta_*^{2p},
\end{array}
\end{equation}
which leads to the desired bound 
by
exploiting Corollary \ref{p4:cor:prml:max}.
\end{proof}


\section{Deterministic supremum bounds}
\label{p4:sec:supbd}

Our goal here is to obtain pathwise bounds on  the
deterministic integrals
\begin{equation}
\label{p4:eq:det:def:if:ibd}
\begin{array}{lcl}
   \mathcal{I}_F(t)
   &= &
     \int_0^te^{-\e(t-s)}\int_0^s\ip{S(s-s')V(s'),S(s-s')F(s')}_{H^1} \, ds' \, ds ,
\\[0.4cm]
  \mathcal{I}^{\mathrm{d}}_B(t)
 & = & 
\int_0^t e^{-\e(t-s)}
\int_0^{s}\nrm{S(s-s')B(s')}^2_{HS(L_Q^2,H^1)} \, d s' \, ds.
\end{array}
\end{equation}
As before, the parameter $\e \in (0, \beta)$ is taken to be fixed throughout the entire section.
We are using the process $F$ in the first integral as a placeholder for various linear and nonlinear expressions in $V$ that we will encounter in {\S}\ref{p4:sec:nls}. The second integral arises 
as the It\^o correction term coming from the integrated $H^1$-norm of $V$; see Lemma \ref{p4:lem:nls:H1:mr}. Besides the assumptions (hB) and (hV) introduced
in {\S}\ref{p4:sec:nls}, we impose the following condition on the new function $F$.
\begin{itemize}
    \item[(hF)] The process $F:[0,T]\times\Omega\to L^2$ 
    has paths in $L^1([0,T];L^2)$ and satisfies
    \begin{align}
        \label{p4:eq:supb:F:psi:tw}
         \ip{F(t),\psi_{\mathrm{tw}}}_{L^2}=0
      \end{align}
      for all $t \in [0,T]$.
\end{itemize}

In contrast to the stochastic setting of {\S}\ref{p4:sec:supb}, pathwise bounds for the expressions \sref{p4:eq:det:def:if:ibd} can be easily used to control
their supremum expectations. Indeed, we do not need to use the Burkholder-Davis-Gundy inequalities, which allows
us to take a far more direct approach to establish our two main
results below. Notice that we are making no a priori assumptions
on the size of $F$. This will be useful in {\S}\ref{p4:sec:nls}
to obtain sharp estimates for the nonlinear terms.

\begin{prop}
\label{p4:prop:det:supb:idf}
There exists a constant $K > 0$ so that
for any  $T>0$,
any pair of processes $(F,V)$ that satisfies (hF) and (hV)
and any $p \ge 1$,
we have the supremum bound
\begin{align}
\label{p4:eq:supb:mr:ib:detF}
E \sup_{0 \leq t \leq T } \abs{ \mathcal{I}_{F}(t)  }^{p}
  &\leq K^{p} \Lambda_*^{p} E \sup_{0 \leq t \leq T}
   \Big[\int_0^t e^{-\e(t-s)} \norm{F(s)}_{L^2} \, ds \Big]^{p}.
\end{align}
\end{prop}

\begin{prop}
\label{p4:prop:det:supb:idb}
There exists a constant $K > 0$ so that 
for any  $T>0$,
any process $B$ that satisfies (hB)
and any $p \ge 1$,
we have the supremum bound
\begin{align}
\label{p4:eq:supb:mr:ib:detd}
E \sup_{0 \leq t \leq T } \mathcal{I}^\mathrm{d}_{B}(t)^{p} 
  &\leq K^p \Theta_*^{2p} .
\end{align}
\end{prop}

\subsection{Estimates for $\mathcal{I}_F$ and $\mathcal{I}_B^\mathrm{d}$}
Upon introducing the expressions
\begin{equation}
\begin{array}{lcl}
    \mathcal{K}_F(t,s') & = & \int_{s'}^{t}e^{-\e(t-s)}\ip{S(s-s')V(s'),S(s-s')F(s')}_{H^1}ds,
\\[0.4cm]
\mathcal{K}_{B;k}^\mathrm{d}(t,s')
&=& \int_{s'}^{t}e^{-\e(t-s)}\ip{S(s-s')B(s')\sqrt{Q} e_k ,S(s-s')B(s') \sqrt{Q} e_k}_{H^1}ds,
\end{array}
\end{equation}
we may reverse the order of integration to find
\begin{equation}
\label{p4:eq:st:bnd:i:for:i:F:t}
    \mathcal{I}_F(t)
    =  \int_0^{t}
       \mathcal{K}_F(t,s') ds',
     \qquad \qquad  
    \mathcal{I}_B^\mathrm{d}(t)
    =  \int_0^{t}
       \sum_{k} \mathcal{K}_{B;k}^\mathrm{d}(t,s') ds' .
\end{equation}

\begin{lem}
There exists a constant $K_{F} \ge 1$ so that
for any 
$T  >0$,
any pair of processes $(V,F)$ that satisfies (hV) and (hF)
and
any $t \in [0,T]$, we have the bound
\begin{equation}
\label{p4:eq:det:f:bnd:for:k:f}
\begin{array}{lcl}  
\mathcal{K}_F(t,s')
& \leq & K_{F} e^{- \e (t-s')} \nrm{V(s')}_{L^2}
\nrm{F(s')}_{L^2}
\end{array}
\end{equation}
for all $0 \leq s' \leq t$.
\end{lem}
\begin{proof}
Observe that $\mathcal{K}_F(t,s')$ is identical
to \sref{p4:eq:st:def:cal:k:i} after making the replacement
$B(s')\sqrt{Q}e_k \mapsto F(s')$.
We can hence use the same decomposition as in {\S}\ref{p4:sec:supb:ib}
and follow the proof of Lemma \ref{p4:lem:st:bnds:cal:k:sq}
to obtain the stated bound.
\end{proof}

\begin{lem}
There exists a
constant $K_B^\mathrm{d} >0$ 
so that
for any 
$T >0$,
any process B that satisfies (hB)
and
any $t \in [0,T]$, we have the bound
\begin{equation}
\label{p4:eq:det:f:bnd:for:k:b:d}
\begin{array}{lcl}  
\sum_{k} \mathcal{K}_{B;k}^\mathrm{d}(t,s')
& \leq & K_B^\mathrm{d} \, e^{- \e (t-s')} \nrm{B(s')}_{HS}^2
\end{array}
\end{equation}
for all $0 \leq s' \leq t$.
\end{lem}
\begin{proof}
Observe that $\mathcal{K}_{B;k}^\mathrm{d}(t,s')$ is identical
to \sref{p4:eq:st:def:cal:k:i} after making the replacement
$V(s')\hspace{-1mm} \mapsto B(s')\sqrt{Q} e_k$.
We can hence use the same decomposition as in {\S}\ref{p4:sec:supb:ib}
and follow the proof of Lemma \ref{p4:lem:st:bnds:cal:k:sq}
to obtain the stated bound.
\end{proof}

\begin{proof}[Proof of Proposition \ref{p4:prop:det:supb:idf}]
Combining the identity \sref{p4:eq:st:bnd:i:for:i:F:t}
with the bound \sref{p4:eq:det:f:bnd:for:k:f},
we readily obtain the pathwise bound
\begin{align}
    \abs{\mathcal{I}_F(t)}
    \leq
    K_F \Lambda_* \int_0^te^{-\e(t-s)}\nrm{F(s)}_{L^2}ds.
\end{align}
The result hence follows by taking the expectation 
of the supremum.
\end{proof}

\begin{proof}[Proof of Proposition \ref{p4:prop:det:supb:idb}]
Combining the identity \sref{p4:eq:st:bnd:i:for:i:F:t}
with the estimate \sref{p4:eq:det:f:bnd:for:k:b:d},
we readily obtain the pathwise bound
\begin{align}
    \abs{\mathcal{I}_B^\mathrm{d}(t)}
    \leq
    K_B^\mathrm{d} \,  \Theta_*^2,
\end{align}
which of course survives taking the expectation 
of the supremum.
\end{proof}


\section{Nonlinear stability }
\label{p4:sec:nls}
With the results from the previous sections under our belt, we now set out to
prove the estimates in Theorem \ref{p4:thm:mr} and hence
establish the stochastic stability of the traveling wave on
exponentially long timescales. Our starting point will
be the computations in   \cite{hamster2017,hamster2020},
which use a time transformation to construct
a mild integral equation for the perturbation $V(t)$ 
that contains no dangerous second order derivatives.

In order to set the stage, we consider
pairs $(U,\Gamma) \in \mathcal{U}_{H^1} \times \mathbb{R}$ for which $\norm{U - \Phi_0(\cdot - \Gamma)}_{L^2}$ is sufficiently small
and introduce the scalar Hilbert-Schmidt operator
\begin{equation}
\label{eq:nls:def:b}
    \overline{b}(U,\G): L^2_Q \ni w \mapsto
- \langle \partial_\xi U ,  \psi_{\mathrm{tw}}(\cdot - \Gamma) \rangle_{L^2}^{-1} 
    \bip{ g(U)v, \psi_{\mathrm{tw}}(\cdot - \Gamma) }_{L^2} ,
\end{equation}
together with the $H^{-1}$-valued expression
\begin{equation}
\begin{array}{lcl}
\mathcal{K}_{\sigma}(U, \Gamma ,c)
& =&
   \rho U''+ c U' + f(U)
     +\frac{1}{2} \sigma^2 \nrm{\overline{b}(U,\G)}^2_{HS}
     U''
\\[0.2cm]
& & \qquad
 - \sigma^2    
 \langle \partial_\xi U ,  \psi_{\mathrm{tw}}(\cdot  - \Gamma) \rangle_{L^2}^{-1} 
   \Big[ g(U) Q g(U) \psi_{\mathrm{tw}}(\cdot - \Gamma)  \Big]' .
\end{array}
\end{equation}
Notice that $\mathcal{K}_{0}(\Phi_0, 0, c_0) = 0$
by construction. For small $\sigma$, one can use the implicit function theorem to define pairs $(\Phi_\s, c_\s)$ that satisfy  $\mathcal{K}_{\sigma}(\Phi_\sigma, 0, c_\sigma) =0$; see \cite[Prop. 5.1]{hamster2020}. With these pairs in hand, one can introduce a second scalar expression
\begin{equation}
\label{eq:nls:def:a}
    \overline{a}_{\s}(U,\G)=
    - \langle \partial_\xi U ,  \psi_{\mathrm{tw}}(\cdot - \Gamma) \rangle_{L^2}^{-1} 
    \ip{   \mathcal{K}_{\s}( U, \G ,c_\sigma), \psi_{\mathrm{tw}}(\cdot - \Gamma)
     }_{L^2}.
\end{equation}
In \cite[{\S}5.2]{hamster2020} we explain how the
definitions \sref{eq:nls:def:b}-\sref{eq:nls:def:a}
can be extended to all $(U, \Gamma) \in \mathcal{U}_{H^1} \times \mathbb{R}$
by utilizing a pair of cut-off operators.

The stochastic phase $\Gamma(t)$ referred to
in Theorem \ref{p4:thm:mr}
is given by
coupling the SDE
\begin{align}\label{eq:res:Gamma}
    d \Gamma=  \big[ c_{\sigma} + \overline{a}_{\sigma}(U,\G) \big] \, dt
     + \sigma \overline{b}(U,\G) dW^Q_t
\end{align}
to our original SPDE \sref{p4:eq:mr:main:spde},
which is well-defined on account of \cite[Prop. 5.2]{hamster2020}. The initial phase $\Gamma(0)$ is chosen in such a way that the perturbation
\begin{equation}
V(t) =  U(\cdot +\G(t),t) - \Phi_\s    
\end{equation}
defined in \sref{p4:eq:int:perturbations}
satisfies $\langle V(0), \psi_{\mathrm{tw}} \rangle_{L^2} = 0$; see \cite[Prop 2.3]{hamster2017}. A delicate argument based on It{\^o}'s lemma subsequently shows  \cite[Prop. 5.4]{hamster2020} that $V$ satisfies the integral identity
\begin{equation}
   \label{eq:stb:eqn:for:V}
     \begin{array}{lcl}
  V(t) & = & V(0) + \int_0^t
     \mathcal{R}_{\sigma}\big( V(s) \big)\, ds
    + \sigma \int_0^t
        \mathcal{S}_{\s}\big(V(s)\big)
       d W^Q_s 
 \end{array}
 \end{equation}
 posed in $H^{-1}$, in which the nonlinearities are given by
\begin{equation}
\begin{array}{lcl}
    \mathcal{R}_\s(V) 
    & = & \mathcal{K}_\s(\Phi_\s + V, 0, c_\s) +  \overline{a}_\s( \Phi_\s +V, 0) [ \Phi_\s' + V' ],
    \\[0.2cm]
    \mathcal{S}_\s(V)[w]
    &=&
     g(\Phi_\s+V)[w]
       +\p_\xi(\Phi_\s+V)\overline{b}(\Phi_{\s} + V,0)[w] 
\end{array}
\end{equation}
for $V \in H^1$ and $w \in L^2_Q$.
By construction, we have $\mathcal{R}_{\sigma}(0) = 0$ and the identity $\langle V(t), \psi_{\mathrm{tw}} \rangle = 0$ is preserved as long as
$\norm{V(t)}_{L^2}$ remains small.

Upon introducing the notation
\begin{equation}
    \kappa_{\sigma}(V) = 1 + \frac{\sigma^2}{2 \rho} \nrm{\overline{b}(\Phi_\s + V,\G)}^2_{HS},
\end{equation}
we observe that $\mathcal{R}_{\sigma}(V)$ contains troublesome nonlinear terms of the form $\rho \kappa_{\sigma}(V) V''$ that lack sufficient smoothness. The situation can be repaired by passing to a transformed time that satisfies $\tau'(t) = \kappa_{\sigma}\big(V(t) \big)$.
Indeed, using \cite[Prop. 6.3]{hamster2017}
we see that 
the transformed function $\overline{V}\big(\tau(t) \big)= V(t)$
admits the mild representation
\begin{equation}
  \label{eq:stb:spde:for:ovl:v}
  \begin{array}{lcl}
    \overline{V}(\tau)
     & = & S(\tau)\overline{V}(0)
     + \int_0^\tau
     S(\tau - \tau') \overline{\mathcal{W}}_{\sigma}\big( \overline{V}(\tau') \big)
            \, d \tau'
     + \sigma \int_0^\tau 
        S(\tau - \tau')
        \overline{\mathcal{S}}_{\sigma}
        \big( \overline{V}(\tau') \big)
       \, d \overline{W}^Q_{\tau'}.
  \end{array}
  \end{equation}
Here $\overline{W}_{\tau}^Q$ represents
  the natural time-transformation of $W_{t}^Q$ (see \cite[Lem. 6.3]{hamster2020}), while 
the nonlinearities are given by
\begin{equation}
\begin{array}{lcl}
\overline{\mathcal{W}}_{\sigma}(V)
& = & 
    \kappa_{\s}^{-1}\big(V \big)
         \mathcal{R}_{\sigma}\big(V\big)
         - \mathcal{L}_{\mathrm{tw}}V ,
\\[0.2cm]
 \overline{\mathcal{S}}_{\sigma}
        \big( V\big)[w]
        & = & \kappa_{\s}^{-1/2}\big(V\big) \mathcal{S}_{\sigma}\big(V\big)[w]
\end{array}
\end{equation}
for $V \in H^1$ and $w \in L^2_Q$. The key point is
that $\overline{\mathcal{W}}_{\s}$ no longer includes second derivatives and hence maps $H^1$ into $L^2$.

The arguments in \cite[Prop. 6.4]{hamster2017}
and \cite[\S 6.1]{hamster2020} indicate that this time transformation
only affects the constants in the final estimate 
\sref{p4:eq:mr:estimate:on:prob}. For presentation purposes,
we therefore simply reuse $t$ and $V$ for the transformed time and perturbation
and leave the definition \sref{p4:eq:mr:defStopTime}
for the stopping time $t_{\mathrm{st}}(\eta)$ intact. 
In order to highlight only the most relevant 
$\sigma$-dependencies and split off linear terms, we also replace 
\sref{eq:stb:spde:for:ovl:v} by the generic system
\begin{align}
\label{p4:eq:nls:ttV}
V(t)=S(t) V(0)+\int_0^tS(t-s)[\s^2F_\mathrm{lin}\big(V(s)\big)+F_\mathrm{nl}\big(V(s)\big)]ds +\s\int_0^tS(t-s)B\big(V(s)\big)dW^Q_s.
\end{align}
Based on the estimates
for $\overline{\mathcal{W}}_{\s}$
and $\overline{\mathcal{S}}_{\s}$
in \cite[App. A]{hamster2020}
and \cite[Prop 8.1]{hamster2017},
we assume that the maps
\begin{equation}
F_{\mathrm{lin}} : H^1 \to L^2,
\qquad
F_{\mathrm{nl}} : H^1 \to L^2,
\qquad
B: H^1 \to HS(L^2_Q, L^2)
\end{equation}
satisfy the bounds 
\begin{equation}
\begin{array}{lcl}
\nrm{F_{\mathrm{lin}}(v)}_{L^2}&\leq& K_{\mathrm{lin}}\nrm{v}_{H^1},
  \\[0.2cm]
\nrm{F_{\mathrm{nl}}(v)}_{L^2}&\leq & K_{\mathrm{nl}}\nrm{v}^2_{H^1}(1+\nrm{v}^3_{L^2}),
  \\[0.2cm]
\nrm{B(v)}_{HS}&\leq&  K_B(1+\nrm{v}_{H^1}) ,
  \\[0.2cm]
\nrm{S(1)B(v)}_{HS}&\leq&  K_BM(1+\nrm{v}_{L^2}).
\end{array}
\end{equation}
In addition, we assume that there exists a constant $\eta_0$ so that
the identities
\begin{equation}
\label{p4:eq:nls:projs:are:zero}
\langle \sigma^2 F_{\mathrm{lin}}(v) + F_{\mathrm{nl}}(v) , \psi_{\mathrm{tw}} \rangle_{L^2}
 = 0,
\qquad \qquad
\langle B(v)[w] , \psi_{\mathrm{tw}} \rangle_{L^2}
 = 0
\end{equation}
hold for every $w \in L^2_Q$
whenever $\norm{v}_{L^2}^2 < \eta_0$. Finally, we assume that $\langle V(0) , \psi_{\mathrm{tw}} \rangle_{L^2} = 0$.


In order to state the main result of this section, we again fix $\e \in (0, \beta)$ and
write
\begin{align}
  N(t)=\norm{V(t)}_{L^2}^2 + \int_0^t e^{-\e(t-s)} \norm{V(s)}_{H^1}^2 \, ds
\end{align}
for the size of the solution $V$ to \sref{p4:eq:nls:ttV}, which also features in the definition
\sref{p4:eq:mr:defStopTime} for the stopping time $t_{\mathrm{st}}=t_{\mathrm{st}}(\eta)$.
The various supremum bounds derived in \S\ref{p4:sec:supb} and \S\ref{p4:sec:supbd} can now be used to obtain
similar bounds for $N(t)$. This result can be seen 
as a significantly sharpened version of its counterpart \cite[Prop. 9.1]{hamster2017}.

\begin{prop}
\label{p4:prp:nls:general}
Pick two sufficiently small
constants
$\delta_{\eta} > 0$
and $\delta_{\sigma} > 0$.
Then there exists
a constant $K > 0$
so that for any integer $T \geq 2$,
any $0 < \eta < \delta_{\eta}$,
any $0 \leq \sigma \leq \delta_{\sigma}$ 
and any $p \ge 1$
we have the bound
\begin{equation}
\label{p4:eq:nls:prp:general:estimate}
E \, \left[\sup_{0\leq t\leq t_{\mathrm{st}}}N(t)^p\right]
\leq K\left[\nrm{V(0)}_{L^2}^{2p}
+\s^{2p}\big(p^p + \ln(T)^p\big) +
 \s^{p} \eta^{p/2} \big( p^{p/2} + \ln(T)^{p/2}\big)\right].
\end{equation}
\end{prop}

The control on all the moments of $N(t)$ allows
us to significantly improve the bound
obtained in \cite[Cor. 9.3]{hamster2017}.
Indeed, we can now employ an exponential Markov inequality to show that the time $T$ can be chosen
to be exponentially large in $1/\sigma$. This allows us establish Theorem \ref{p4:thm:mr}
in a standard fashion.

\begin{cor}
\label{p4:cor:nls:bnd}
Under the conditions of Proposition
\ref{p4:prp:nls:general}, 
we have the bound
\begin{equation}
\begin{array}{lcl}
P(t_{\mathrm{st}}(\eta) < T ) 
& \le & 
3 T^{\frac{1}{2e}} \mathrm{exp}[
- \frac{ \eta - 2eK \norm{V(0)}^2_{L^2} }
{ {2eK \sigma (\sigma + \sqrt{\eta})} }
].
\end{array}
\end{equation}
\end{cor}
\begin{proof}
For convenience, we introduce the random variable
\begin{equation}
    Z = \sup_{0 \le t \le t_{\mathrm{st}} } N(t).
\end{equation}
Picking an arbitrary $\nu > 0$, 
we observe that
\begin{equation}
\begin{array}{lcl}
p(t_{\mathrm{st}}(\eta) < T ) &=& P\Big(
 \sup_{0 \leq t \leq T} \big[N(t)\big]
 > \eta
\Big)
\\[0.2cm]
& = & 
e^{-\nu \eta} E [ \mathbf{1}_{t_{st} < T}  e^{\nu N(t_{st}) } ]
\\[0.2cm]
&\leq &
     e^{-\nu \eta}
     E e^{\nu N(t_{st})}
    \\[0.2cm]
&\leq &
     e^{-\nu \eta}
     E e^{\nu Z}.
    \\[0.2cm]    
\end{array}
\end{equation} 
Upon introducing the quantities
\begin{equation}
    \Theta_1 = K\sigma(\sigma + \sqrt{\eta}),
    \qquad \qquad
    \Theta_2 = K\big( \sigma(\sigma + \sqrt{\eta}) \ln(T) + \norm{V(0)}_{L^2}^2 \big),
\end{equation}
the bound \sref{p4:eq:nls:prp:general:estimate}
can be simplified (and slightly weakened)
into the form
\begin{equation}
    E Z^p \le p^p \Theta_1^p + \Theta_2^p.
\end{equation}
Using $p! \geq p^{p} e^{-p}$,
this allows us to compute
\begin{equation}
\begin{array}{lcl}
    E e^{\nu Z}
& \le & 
\sum_{p=0}^\infty \frac{\nu^p}{p!}
      \big(p^{p} \Theta_1^p + \Theta_2^p \big)
\\[0.2cm]
& \le &
e^{\nu \Theta_2} +
\sum_{p=0}^\infty 
     \nu^p
     e^{p}\Theta_1^{p}.
\end{array}
\end{equation}
Choosing $\nu = (2e \Theta_1)^{-1}$,
we obtain
\begin{equation}
\begin{array}{lcl}
p(t_{\mathrm{st}}(\eta) < T ) 
& \le & \mathrm{exp}\big[- \frac{\eta}{2eK \sigma (\sigma + \sqrt{\eta})}\big]
\Big(
2 + \mathrm{exp}\big[
\frac{\ln(T)}{2e} +  \frac{\norm{V(0)}^2_{L^2}}{\sigma(\sigma + \sqrt{\eta}) }
\big] 
\Big),
\\[0.2cm]
\end{array}
\end{equation}
which can be absorbed into the stated bound.
\end{proof}

\begin{proof}[Proof of Theorem \ref{p4:thm:mr}]
Following the proof of \cite[Thm. 2.4]{hamster2017}, we can undo the stochastic time-transformation.  The desired bound
then follows  immediately from Corollary
\ref{p4:cor:nls:bnd} upon choosing $\kappa\leq (4eK)^{-1}$ and noting that $3T^\frac{1}{2e}<2T$ for $T\geq 2$. 
%
\end{proof}

\subsection{Supremum bounds in $L^2$}
\label{p4:sec:nls:L2}

In this subsection we establish the following  bound on 
the supremum of the $L^2$-norm of $V(t)$. Notice
that we are imposing less restrictions on $\sigma$ and $\eta$ here,
but $N(t)$ still appears on the right-hand side
of our estimate.
\begin{lem}
\label{p4:lem:nls:L2:mr}
There exists a constant $K \ge 1$
so that for any integer $T \geq 2$, any $0 < \eta < \eta_0$, any $0 \leq \sigma \leq 1$
and any $p \ge 1$,
we have the bound
\begin{align}
    E\sup_{0\leq t\leq t_\mathrm{st}}\nrm{V(t)}_{L^2}^{2p}\leq K^{2p}\Big[\nrm{V(0)}_{L^2}^{2p}+\s^{2p}\ln(T)^{p} + \s^{2p} p^{p}
    +(\s^4+\eta)^{p} E\sup_{0\leq t\leq t_\mathrm{st}}N(t)^{p}\Big].
\end{align}
\end{lem}

In order to streamline the proof, we recall the definition of $\Pi$ from Lemma \ref{p4:lem:prml:semigroup} and define the  functions
\begin{align}\begin{split}
 \mathcal{E}_0(t)=&S(t)V(0),\\[0.2cm]
 \mathcal{E}_\mathrm{lin}(t)=&\int_0^tS(t-s)\Pi F_\mathrm{lin}\big(V(s)\big)
    \mathbf{1}_{s\leq t_\mathrm{st}} ds,\\[0.2cm]
 \mathcal{E}_\mathrm{nl}(t)=&\int_0^tS(t-s)\Pi F_\mathrm{nl}\big(V(s)\big)
    \mathbf{1}_{s\leq t_\mathrm{st}}ds,\\[0.2cm]
 \mathcal{E}_{B}(t)=&\int_0^tS(t-s)B\big(V(s)\big)
    \mathbf{1}_{s\leq t_\mathrm{st}}dW_s^Q.
\end{split}
\end{align}
The three deterministic expressions can be controlled in a direction fashion,
while the final stochastic integral was analyzed in {\S}\ref{p4:sec:supbd}.

\begin{lem}
\label{p4:lem:nls:supbL2:Edet}
For any $0 < \eta < \eta_0$, any $0 \leq \sigma \leq 1$, any $T>0$ and any $p \ge 1$,
we have the bounds
\begin{align}
\begin{split}
    E\sup_{0\leq t\leq T} \nrm{\mathcal{E}_0(t)}_{L^2}^{2p}\leq\,
    & M^{4p}\nrm{V(0)}^{2p}_{L^2},\\[0.2cm]
    E\sup_{0\leq t\leq T}\nrm{\mathcal{E}_\mathrm{lin}(t)}_{L^2}^{2p}\leq\, &M^{2p}K^{2p}_{\mathrm{lin}}E\sup_{0\leq t\leq t_\mathrm{st}}\nrm{V(t)}_{L^2}^{2p},\\
    E\sup_{0\leq t\leq T} \nrm{\mathcal{E}_\mathrm{nl}(t)}_{L^2}^{2p}\leq\, &M^{2p} K^{2p}_{\mathrm{nl}}(1+\eta^3)^{2p}\eta^{p} 
    E\sup_{0\leq t\leq t_\mathrm{st}}
    \Big[\int_0^te^{-\e(t-s)}\nrm{V(t)}_{H^1}^2ds \Big]^{p}.
\end{split}
\end{align}
\end{lem}
\begin{proof}
These results follow directly from Lemmas 9.8-9.11 in \cite{hamster2017},
where they were established using straightforward pathwise norm estimates.  
\end{proof}

\begin{lem}
\label{p4:lem:nls:supbL2:Est}
There is a $K \ge 1$ such that the bound
\begin{align}
    E\sup_{0\leq t\leq T}\nrm{\mathcal{E}_{B}(t)}_{L^2}^{2p}\leq \big( p^{p} + \ln(T)^{p} \big) K^{2p} 
\end{align}
holds for any integer $T \geq 2$, any $0 < \eta < \eta_0$, any $0 \leq \sigma \leq 1$
and any $p \ge 1$.
\end{lem}

\begin{proof}
We will prove this by appealing to Proposition \ref{p4:prop:supb:mr}.
In order to verify (hB), we simply compute
\begin{align}
\begin{split}
    \int_0^te^{-\e(t-s)}\nrm{ B\big(V(s)\big)\mathbf{1}_{s\leq t_\mathrm{st}}}_{HS}^2 \, ds
    \leq\,& K^2_B\int_0^te^{-\e(t-s)}(1+\nrm{V(s)}_{H^1}^2) 
    \mathbf{1}_{s\leq t_\mathrm{st}} \, ds\\[0.1cm]
    \leq\,&K^2_B\left(\e^{-1}+\int_0^{\min\{t, t_{\mathrm{st}} \} } e^{-\e(t-s)}\nrm{V(s)}_{H^1}^2
     \,
    ds\right)
    \\[0.1cm]
    \leq& K^2_B(\e^{-1} + \eta),
\end{split}
\end{align}
together with
\begin{align}
\begin{split}
\norm{S(1) B\big(V(s)\big)\mathbf{1}_{s\leq t_\mathrm{st}}}_{HS}^{2}\leq& M^2 K_B^2(1+\nrm{V(s)\mathbf{1}_{s\leq t_\mathrm{st}}}^2_{L^2})
\leq  M^2K_B^2(1+\eta),
\end{split}
\end{align}
which allows us to take $\Theta_*^2=M^2K_B^2(\e^{-1}+\eta)$.
\end{proof}

\begin{proof}[Proof of Lemma \ref{p4:lem:nls:L2:mr}]
We directly find that
\begin{align}
    E\sup_{0\leq t\leq t_\mathrm{st}}\nrm{V(t)}_{L^2}^{2p}\leq 2^{2p} E\sup_{0\leq t\leq T}\big[ \nrm{\mathcal{E}_0(t)}_{L^2}^{2p}+
 \s^{4p} \nrm{\mathcal{E}_\mathrm{lin}(t)}_{L^2}^{4p}+
 \nrm{\mathcal{E}_\mathrm{nl}(t)}_{L^2}^{2p}+\s^{2p} \nrm{\mathcal{E}_{B}(t)}_{L^2}^{2p}
 \big].
\end{align}
Collecting the results from Lemmas
\ref{p4:lem:nls:supbL2:Edet} and  \ref{p4:lem:nls:supbL2:Est}
now proves the result. 
\end{proof}

\subsection{Supremum bounds in $H^1$}
\label{p4:sec:nls:H1}
In this subsection we control the $H^1$-norm of $V$
by establishing a supremum bound
for the integrated expression
\begin{align}
    \mathcal{I}(t)=\int_0^t e^{-\e(t-s)}\nrm{V(s)}^2_{H^1} ds.
\end{align}
In particular, we set out to obtain the following estimate.

\begin{lem}
\label{p4:lem:nls:H1:mr}
There exists a constant $K \ge 1$
so that for any integer $T \geq 2$, any $0 < \eta < \eta_0$, any $0 \leq \sigma \leq 1$
and any $p \ge 1$
we have the bound
\begin{align}
       E\sup_{0\leq t\leq T} \mathcal{I}(t)^p \leq K^{p}\Big[\nrm{V(0)}_{H^1}^{2p}
       +\eta^{p/2}E\big[\sup_{0\leq t\leq t_\mathrm{st}}N(t)^p\big]+\s^{2p}
       +\s^p\eta^{p/2}\sqrt{\ln(T)} 
       + \s^p \eta^{p/2} p^{p/2} \Big].
\end{align}
\end{lem}

Compared to {\S}\ref{p4:sec:nls:L2} and \cite[{\S}9]{hamster2017}, our approach here is rather indirect.
First of all, we
exploit the fact that $T$ is an integer
to compute
\begin{align}
\begin{split}
    \sup_{0 \leq t\leq T}\mathcal{I}(t)
 =&\, \max_{i\in \{1,....,T\}}\sup_{i-1\leq t\leq i}\int_0^te^{-\e(t-s)}\nrm{V(s)}^2_{H^1}  ds\\[0.2cm]
\leq&\,  \max_{i\in \{1,....,T\}}
e^{\e}\int_0^ie^{-\e(i-s)}\nrm{V(s)}^2_{H^1}  ds\\[0.2cm]
=&\, e^{\e} \max_{i\in \{1,....,T\}}\mathcal{I}(i).
\end{split}
\end{align}
We continue by applying
a mild It\^o formula \cite{dapratomild} to $\nrm{V(s)}_{H^1}^2$,
which yields
\begin{align}
\begin{split}
    \nrm{V(s)}_{H^1}^2=&\nrm{S(s)V(0)}_{H^1}^2+2\s^2\int_0^s\ip{S(s-s')V(s'),S(s-s')F_\mathrm{lin}(V(s'))}_{H^1}ds'\\[0.2cm]
    &+2\int_0^s\ip{S(s-s')V(s'),S(s-s')F_\mathrm{nl}\big(V(s')\big)}_{H^1}ds'\\[0.2cm]
    &+2\s\int_0^s\ip{S(s-s')V(s'),S(s-s')B\big(V(s')\big)dW^Q_{s'}}_{H^1}\\[0.2cm]
    &+\s^2\int_0^{s}\nrm{S(s-s')B\big(V(s')\big)}^2_{HS(L_Q^2,H^1)}ds'.
\end{split}
\end{align}
In particular, upon introducing the components 
\begin{align}
\begin{split}
\mathcal{I}_0(t)=& \int_0^t e^{-\e(t-s)}\nrm{S(s)V(0)}_{H^1}^2ds,\\[0.2cm]
\mathcal{I}_{\mathrm{lin}}(t)=&\int_0^te^{-\e(t-s)}\int_0^s\ip{S(s-s')V(s'),S(s-s')\Pi F_\mathrm{lin}\big(V(s')\big)\mathbf{1}_{s'\leq t_\mathrm{st}}}_{H^1}ds'ds,\\[0.2cm]
\mathcal{I}_{\mathrm{nl}}(t)=&\int_0^te^{-\e(t-s)}\int_0^s\ip{S(s-s')V(s'),S(s-s')\Pi F_\mathrm{nl}\big(V(s')\big)\mathbf{1}_{s'\leq t_\mathrm{st}}}_{H^1}ds'ds,\\[0.2cm]
\mathcal{I}^{\mspace{1mu}\mathrm{s}}_{B}(t)=&\int_0^ie^{-\e(t-s)}\int_0^s\ip{S(s-s')V(s'),S(s-s')B\big(V(s')\big)\mathbf{1}_{s'\leq t_\mathrm{st}}dW^Q_{s'}}_{H^1}ds,\\[0.2cm]
\mathcal{I}^\mathrm{d}_{B}(t)=&\int_0^ie^{-\e(t-s)}\int_0^s\nrm{S(s-s')B\big(V(s')\big)\mathbf{1}_{s'\leq t_\mathrm{st}}}^2_{HS(L_Q^2,H^1)}ds'ds
\end{split}
\end{align}
and applying Jensen's inequality,
we obtain the bound
\begin{align}
\label{p4:eq:nls:splitIbbound}
\begin{split}
   E\sup_{0\leq t\leq t_\mathrm{st}} \mathcal{I}(t)^p \leq&\,
     e^{\e p}5^{p-1}E\max_{i\in \{1,....,T\}} 
    \left[\mathcal{I}_0(i)^p+2^p\s^{2p}\mathcal{I}_{\mathrm{lin}}(i)^p+2^p\mathcal{I}_{\mathrm{nl}}(i)^p
    +2^p\s^p\mathcal{I}^{\mspace{1mu}\mathrm{s}}_B(i)^p
    +\s^{2p}\mathcal{I}^{\mathrm{d}}_B(i)^p
    \right]\\[0.2cm]
    \leq&\,
     10^p e^{\e p}E\sup_{0\leq t\leq T}
    \left[\mathcal{I}_0(t)^p+\s^{2p}\mathcal{I}_{\mathrm{lin}}(t)+\mathcal{I}_{\mathrm{nl}}(t)^p
        +\s^{2p}\mathcal{I}^{\mathrm{d}}_B(t)^p
    \right]
    \\[0.2cm]
    & \qquad \qquad +10^p e^{\e p}\s^p E\max_{i\in \{1,....,T\}}\mathcal{I}^{\mspace{1mu}\mathrm{s}}_B(i)^p.
\end{split}
\end{align}
This decomposition highlights the fact that supremum bounds over deterministic integrals are easily obtained, while the stochastic integral needs
to be handled with care.

\begin{lem}
\label{p4:lem:nls:bnd:ilin:inl}
There exists a constant $K \ge 1$
so that for any integer $T \geq 2$, any $0 < \eta < \eta_0$, any $0 \leq \sigma \leq 1$
and any $p \ge 1$
we have the bounds
\begin{align}
\begin{split}
    E\sup_{0\leq t\leq T} \mathcal{I}_{\mathrm{lin}}(t)^p \leq &\, K^p\eta^p, \\
    E\sup_{0\leq t\leq T} \mathcal{I}_{\mathrm{nl}}(t)^p   \leq&\, 
    K^p \eta^{p/2} E\sup_{0\leq t\leq T} \Big[\int_0^te^{-\e(t-s)}\nrm{V(s)}_{H^1}^2\mathbf{1}_{s\leq t_\mathrm{st}}ds \Big]^p.
\end{split}
\end{align}
\end{lem}
\begin{proof}
In order to exploit Proposition \ref{p4:prop:det:supb:idf},
we first note that the orthogonality conditions
\sref{p4:eq:st:ip:v:psi:tw} and
\sref{p4:eq:supb:F:psi:tw} hold true by virtue of the stopping time.
In particular, (hF) and (hV) are both satisfied, 
with $\Lambda_* = \sqrt{\eta}$.
The stated bounds can hence
be obtained by using 
the computation
\begin{align}
\begin{split}
    \int_0^te^{-\e(t-s)}\nrm{F_\mathrm{lin}(V(s))\mathbf{1}_{s\leq t_\mathrm{st}}}_{L^2}ds\leq\,
    & K_{\mathrm{lin}}\int_0^te^{-\e(t-s)}\nrm{V(s)}_{H^1}\mathbf{1}_{s\leq t_\mathrm{st}}ds\\
    \leq\, & K_{\mathrm{lin}} \frac{1}{\sqrt{\e}}\sqrt{\int_0^te^{-\e(t-s)}\nrm{V(s)}_{H^1}^2\mathbf{1}_{s\leq t_\mathrm{st}}ds }
    \\
    \leq\, & K_{\mathrm{lin}} \sqrt{\frac{\eta}{\e}},
\end{split}
\end{align}
together with
\begin{align}
\begin{split}
    \int_0^te^{-\e(t-s)}\nrm{F_\mathrm{nl}(V(s))\mathbf{1}_{s\leq t_\mathrm{st}}}_{L^2}ds\leq\,
    & K_{\mathrm{nl}}\int_0^te^{-\e(t-s)}\nrm{V(s)}_{H^1}^2(1+\nrm{V(s)}_{L^2}^3)\mathbf{1}_{s\leq t_\mathrm{st}}ds\\[0.1cm]
    \leq\, & K_{\mathrm{nl}}(1+\eta^3)\int_0^te^{-\e(t-s)}\nrm{V(s)}_{H^1}^2\mathbf{1}_{s\leq t_\mathrm{st}}ds
\end{split}
\end{align}
to evaluate the right-hand side 
of \sref{p4:eq:supb:mr:ib:detF}.
\end{proof}

\begin{lem}
\label{p4:lem:nls:bnd:ibd:ibs}
There exists a constant $K \ge 1$
so that for any integer $T \geq 2$, any $0 < \eta < \eta_0$, any $0 \leq \sigma \leq 1$
and any $p \ge 1$
we have the bounds
\begin{align}
\begin{split}
   E\sup_{0\leq t\leq T}  \mathcal{I}_B^{\mathrm{d}}(t)^p \leq &\,  K^p,\\[0.2cm]
   E\max_{i\in \{1,....,T\}}  \mathcal{I}^{\mspace{1mu}\mathrm{s}}_{B}(t)^p\leq &\, K^p \eta^{p/2} \big( p^{p/2} + \ln(T)^{p/2} \big).
\end{split}
\end{align}
\end{lem}
\begin{proof}
Recall from the proof of Lemma \ref{p4:lem:nls:L2:mr} that (hB) holds with  $\Theta_*^2=M^2K_B^2(\e^{-1}+\eta)$.
The first estimate now follows directly from
Proposition \ref{p4:prop:det:supb:idb}, while the second
can be obtained from Proposition 
\ref{p4:prop:supb:mr} using the fact that (hV) is satisfied
with $\Lambda_* = \sqrt{\eta}$.
\end{proof}

\begin{proof}[Proof of Lemma \ref{p4:lem:nls:H1:mr}]
The bound follows immediately from the decomposition
\sref{p4:eq:nls:splitIbbound} and Lemmas  
\ref{p4:lem:nls:bnd:ilin:inl}-\ref{p4:lem:nls:bnd:ibd:ibs}.
\end{proof}

\begin{proof}[Proof of Proposition \ref{p4:prp:nls:general}]
Summing the estimates from Lemmas \ref{p4:lem:nls:L2:mr}
and \ref{p4:lem:nls:H1:mr} yields the bound
\begin{equation}
\begin{array}{lcl}
E \, \big[\sup_{0\leq t\leq t_{\mathrm{st}} }N(t)^p\big]
& \leq & K^p\Big[\nrm{V(0)}_{L^2}^{2p}+\s^{2p}\ln(T)^p
+ \s^{2p} p^p
+\s^p \eta^{p/2} \ln(T)^{p/2}
+ \s^p \eta^{p/2} p^{p/2}
\\[0.2cm]
& & \qquad \qquad
+(\s^{4p}+\eta^{p/2})E\big[
\sup_{0\leq t\leq t_\mathrm{st}}
N(t)^p \big]\Big].
\end{array}
\end{equation}
Upon restricting the size of $\s^4+\sqrt{\eta}$, the result readily 
follows. 
\end{proof}

\bibliographystyle{klunumHJ}
\bibliography{ref}

\end{document}